\documentclass[a4paper,10pt]{article}%
\usepackage{amsmath}%
\usepackage{amsfonts}%
\usepackage{amssymb}%
\usepackage{amsthm}
\usepackage{graphicx}
\usepackage[latin1]{inputenc}
\usepackage[T1]{fontenc}
\usepackage[english,french]{babel}

\addtocounter{equation}{15}

\newtheorem*{theo}{Theorem}

\newtheorem{theorem}{Theorem}[section]

\newtheorem{conjecture}[theorem]{Conjecture}
\newtheorem{corollary}[theorem]{Corollary}

\newtheorem{lemma}[theorem]{Lemma}

\newtheorem{proposition}[theorem]{Proposition}

\begin{document}

\sloppy

\title{Actions of the group of homeomorphisms of the circle on surfaces}
\author{E. Militon \footnote{supported by the Fondation Mathématique Jacques Hadamard}}
\date{\today}
\maketitle

\setlength{\parskip}{10pt}

\selectlanguage{english}
\begin{abstract}
In this article we describe all the group morphisms from the group of orientation preserving homeomorphisms of the circle to the group of homeomorphisms of the annulus or of the torus.
\end{abstract}
\section{Introduction}

For a compact manifold $M$, we denote by $\mathrm{Homeo}(M)$ the group of homeomorphisms of $M$ and by $\mathrm{Homeo}_{0}(M)$ the connected component of the identity of this group (for the compact-open topology). By a theorem by Fischer (see \cite{Fis} and \cite{Bou}), the group $\mathrm{Homeo}_{0}(M)$ is simple: the study of this group cannot be reduced to the study of other groups. One natural way to have a better understanding of this group is to look at its automorphism group. In this direction, Whittaker proved the following theorem.

\begin{theo}{(Whittaker \cite{Whi})}
Given two compact manifolds $M$ and $N$ and for any group isomorphism $\varphi : \mathrm{Homeo}_{0}(M) \rightarrow \mathrm{Homeo}_{0}(N)$, there exists a homeomorphism $h : M \rightarrow N$ such that the morphism $\varphi$ is the map $f \mapsto h \circ f \circ h^{-1}$.
\end{theo}

This theorem was generalized by Filipkiewicz (see \cite{Fil}) to groups of diffeomorphisms by using a powerful theorem by Montgomery and Zippin which characterizes Lie groups among locally compact groups. The idea of the proof of Whittaker's theorem is the following: we see the manifold $M$ algebraically by considering the subgroup $G_{x}$ of the group of homeomorphism of $M$ consisting of homeomorphisms which fix the point $x$ in $M$. One proves that, for any point $x$ in $M$, there exists a unique point $y_{x}$ in $N$ such that $\varphi(G_{x})$ is the group of homeomorphisms of $N$ which fix the point $y_{x}$. Let us define $h$ by $h(x)=y_{x}$. Then we check that $h$ is a homeomorphism and that the homeomorphism $h$ satisfies the conclusion of the theorem. However, Whittaker's proof uses crucially the fact that we have a group isomorphism and cannot be generalized easily a priori to the cases of group morphisms. Here is a conjecture in this case.

\begin{conjecture} \label{Whittaker}
For a compact manifold $M$, every group morphism $\mathrm{Homeo}_{0}(M) \rightarrow \mathrm{Homeo}_{0}(M)$ is either trivial or induced by a conjugacy by a homeomorphism.
\end{conjecture}

This conjecture is solved in \cite{Mat} in the case of the circle. It is proved also in \cite{Man} in the case of groups of diffeomorphisms of the circle. We can also be interested in morphisms from the group $\mathrm{Homeo}_{0}(M)$ to the group of homeomorphisms of another manifold $\mathrm{Homeo}_{0}(N)$. This kind of questions is adressed in \cite{Man} for diffeomorphisms in the case where $N$ is a circle or the real line. The following conjecture looks attainable.

\begin{conjecture}
Denote by $S-D$ a closed surface $S$ with one open disc removed. Denote by $\mathrm{Homeo}_{0}(S-D)$ the identity component of the group of homeomorphisms $S-D$ with support contained in the interior of this surface. Every group morphism from $\mathrm{Homeo}_{0}(S-D)$ to $\mathrm{Homeo}_{0}(S)$ is induced by an inclusion of $S-D$ in $S$.
\end{conjecture}

In this article, we investigate the case of group morphisms from the group $\mathrm{Homeo}_{0}(\mathbb{S}^{1})$ of orientation-preserving homeomorphisms of the circle to the group $\mathrm{Homeo}(S)$ of homeomorphisms of a compact orientable surface $S$. By simplicity of the group $\mathrm{Homeo}_{0}(\mathbb{S}^{1})$, such a morphism $\varphi$ is either one-to-one or trivial. Moreover, as the quotient group $\mathrm{Homeo}(S)/\mathrm{Homeo}_{0}(S)$ is countable, any morphism $\mathrm{Homeo}_{0}(\mathbb{S}^{1}) \rightarrow \mathrm{Homeo}(S)/\mathrm{Homeo}_{0}(S)$ is trivial. Hence, the image of the morphism $\varphi$ is contained in the group $\mathrm{Homeo}_{0}(S)$. By a theorem by Rosendal and Solecki, any group morphism from the group of orientation-preserving homeomorphisms of the circle to a separable group is continuous (see \cite{RS} Theorem 4 and Proposition 2): the group morphisms under consideration are continuous. If the surface $S$ is different from the sphere, the torus, the closed disc or the closed annulus, then any compact subgroup of the group $\mathrm{Homeo}_{0}(S)$ is trivial by a theorem by Kerekjarto (see \cite{Kol}). Therefore, the group $\mathrm{Homeo}_{0}(S)$ does not contain a subgroup isomorphic to $\mathrm{SO}(2)$ whereas the group $\mathrm{Homeo}_{0}(\mathbb{S}^{1})$ does and the morphism $\varphi$ is necessarily trivial. We study in what follows the remaining cases.

In the second section, we state our classification theorem of actions of the group $\mathrm{Homeo}_{0}(\mathbb{S}^{1})$ on the torus and on the annulus. Any action will be obtained by gluing actions which preserve a lamination by circles and actions which are transitive on open annuli. In the third section, we describe the continuous actions of the group of compactly-supported homeomorphisms of the real line on the real line or on the circle: this description is useful for the classification theorem and interesting in itself. The fourth and fifth sections are devoted respectively to the proof of the classification theorem in the case of the closed annulus and in the case of the torus. Finally, we discuss the case of the sphere and of the closed disc in the last section.

\section{Description of the actions}

All the actions of the group $\mathrm{Homeo}_{0}(\mathbb{S}^{1})$ will be obtained by gluing elementary actions of this group on the closed annulus. In this section, we will first describe these elementary actions before describing some model actions to which any continuous action will be conjugate.

The easiest action of the group $\mathrm{Homeo}_{0}(\mathbb{S}^{1})$ on the closed annulus $\mathbb{A}=[0,1] \times \mathbb{S}^{1}$ preserves of a foliation by circles of this annulus:
$$ \begin{array}{rrcl}
p: & \mathrm{Homeo}_{0}(\mathbb{S}^{1}) & \rightarrow & \mathrm{Homeo}(\mathbb{A}) \\
 & f & \mapsto & ((r, \theta) \mapsto (r, f(\theta)))
\end{array}
.
$$

The second elementary action we want to describe is a little more complex. Let $\pi: \mathbb{R} \rightarrow \mathbb{R} /\mathbb{Z}= \mathbb{S}^{1}$ be the projection. For a point $\theta$ on the circle $\mathbb{S}^{1}= \mathbb{R} / \mathbb{Z}$, we denote by $\tilde{\theta}$ a lift of $\theta$, \emph{i.e.} a point of the real line which projects on $\theta$. For a homeomorphism $f$ of the circle, we denote by $\tilde{f}$ a lift of $f$, \emph{i.e.} an element of the group $\mathrm{Homeo}_{\mathbb{Z}}(\mathbb{R})$ of homeomorphisms of the real line which commute to integral translations such that $\pi \circ \tilde{f}=f \circ \pi$. A second elementary action is given by:
$$\begin{array}{rrcl}
a_{-}: & \mathrm{Homeo}_{0}(\mathbb{S}^{1}) & \rightarrow & \mathrm{Homeo}(\mathbb{A}) \\
 & f & \mapsto & ((r, \theta) \mapsto (\tilde{f}(\tilde{\theta})-\tilde{f}(\tilde{\theta}-r), f(\theta))
\end{array}
.$$ 
Notice that the number $\tilde{f}(\tilde{\theta})-\tilde{f}(\tilde{\theta}-r)$ does not depend on the lifts $\tilde{\theta}$ and $\tilde{f}$ chosen and belongs to the interval $[0,1]$. An action analogous to this last one is given by:
$$\begin{array}{rrcl}
a_{+}: & \mathrm{Homeo}_{0}(\mathbb{S}^{1}) & \rightarrow & \mathrm{Homeo}(\mathbb{A}) \\
 & f & \mapsto & ((r, \theta) \mapsto (\tilde{f}(\tilde{\theta}+r)-\tilde{f}(\tilde{\theta}), f(\theta)))
\end{array}
.$$
Notice that the actions $a_{+}$ and $a_{-}$ are conjugate via the homeomorphism of $\mathbb{A}$ given by $(r, \theta) \mapsto (1-r, \theta)$, which is orientation-reversing, and via the orientation-preserving homeomorphism of $\mathbb{A}$ given by $(r, \theta) \mapsto (r, \theta+r)$. 
We now describe another way to see the action $a_{-}$ (and $a_{+}$). We see the torus $\mathbb{T}^{2}$ as the product $\mathbb{S}^{1} \times \mathbb{S}^{1}$. Let
$$\begin{array}{rrcl}
a_{\mathbb{T}^{2}}: & \mathrm{Homeo}_{0}(\mathbb{S}^{1}) & \rightarrow & \mathrm{Homeo}(\mathbb{T}^{2}) \\
 & f & \mapsto & ((x, y) \mapsto (f(x),f(y)))
\end{array}
.$$
It is easily checked that this defines a morphism.
This action leaves the diagonal $\left\{(x,x), x \in \mathbb{S}^{1} \right\}$ invariant. The action obtained by cutting along the diagonal is conjugate to $a_{-}$. More precisely, let us define
$$\begin{array}{rrcl}
h: & \mathbb{A} & \rightarrow & \mathbb{T}^{2} \\
 & (r,\theta) & \mapsto & (\theta, \theta-r)
\end{array}
.$$
Then, for any element $f$ in the group $\mathrm{Homeo}_{0}(\mathbb{S}^{1})$:
$$h \circ a_{-}(f)= a_{\mathbb{T}^{2}}(f) \circ h.$$

Let $G_{\theta_{0}}$ be the group of homeomorphisms of the circle which fix a neighbourhood of the point $\theta_{0}$. For a point $\theta_{0}$ of the circle, the image by the morphism $a_{-}$ of the group $G_{\theta_{0}}$ leaves the sets $\left\{ (r, \theta_{0}), r \in [0,1] \right\}$ and $\left\{ (r, \theta_{0}+r), r \in [0,1] \right\}$ globally invariant. On each of the connected components of the complement of the union of these two sets with the boundary of the annulus, the action of $G_{\theta_{0}}$ is transitive (this is more easily seen by using the action $a_{\mathbb{T}^{2}}$).

Let us describe now the model actions on the closed annulus which are obtained by gluing the above actions. Take a non-empty compact subset $K \subset [0,1]$ which contains the points $0$ and $1$ and a map $\lambda: [0,1]-K \rightarrow \left\{-1,+1 \right\}$ which is constant on each connected component of $[0,1]-K$. Let us define now an action $\varphi_{K, \lambda}$ of the group $\mathrm{Homeo}_{0}(\mathbb{S}^{1})$ on the closed annulus. To a homeomorphism $f$ in $\mathrm{Homeo}_{0}(\mathbb{S}^{1})$, we associate a homeomorphism $\varphi_{K, \lambda}(f)$ which is defined as follows. If $r \in K$, we associate to a point $(r, \theta) \in \mathbb{A}$ the point $(r, f(\theta))$. If $r$ belongs to a connected component $(r_{1}, r_{2})$ of the complement of the compact set $K$ and if $\lambda((r_{1},r_{2}))=\left\{-1 \right\}$, we associate to the point $(r, \theta) \in \mathbb{A}$ the point
$$ ((r_{2}-r_{1})(\tilde{f}(\theta)-\tilde{f}(\theta-\frac{r-r_{1}}{r_{2}-r_{1}}))+r_{1}, f(\theta))).$$
This last map is obtained by conjugating the homeomorphism $a_{-}(f)$ with the map $(r, \theta) \mapsto (\zeta(r), \theta)$ where $\zeta$ is the unique linear increasing homeomorphism $[0,1] \rightarrow [r_{1},r_{2}]$. If $r$ belongs to a connected component $(r_{1}, r_{2})$ of the complement of the compact set $K$ and if $\lambda((r_{1},r_{2}))=\left\{+1 \right\}$, we associate to the point $(r, \theta) \in \mathbb{A}$ the point
$$ ((r_{2}-r_{1})(\tilde{f}(\theta+\frac{r-r_{1}}{r_{2}-r_{1}})-\tilde{f}(\theta))+r_{1}, f(\theta))).$$
This last map is also obtained after renormalizing $a_{+}(f)$ on the interval $(r_{1},r_{2})$. This defines a continuous morphism $\varphi_{K, \lambda} :\mathrm{Homeo}_{0}(\mathbb{S}^{1}) \rightarrow \mathrm{Homeo}(\mathbb{A})$. To construct an action on the torus, it suffices to identify the point $(0, \theta)$ of the closed annulus with the point $(1,\theta)$. We denote by $\varphi_{K, \lambda}^{\mathbb{T}^{2}}$ the continuous action on the torus obtained this way. By shrinking one of the boundary components (respectively both boundary components) of the annulus to a point (respectively to points), one obtains an action on the closed disc (respectively on the sphere) that we denote by $\varphi_{K,\lambda}^{\mathbb{D}^{2}}$ (respectively $\varphi_{K,\lambda}^{\mathbb{S}^{2}}$).

The main theorem of this article is the following:

\begin{theorem} \label{cerclesursurfaces}
Any non-trivial action of the group $\mathrm{Homeo}_{0}(\mathbb{S}^{1})$ on the closed annulus is conjugate to one of the actions $\varphi_{K,\lambda}$. Any non-trivial action of the group $\mathrm{Homeo}_{0}(\mathbb{S}^{1})$ on the torus is conjugate to one of the actions $\varphi_{K, \lambda}^{\mathbb{T}^{2}}$.
\end{theorem}

In particular, any action of the group $\mathrm{Homeo}_{0}(\mathbb{S}^{1})$ on the torus admits an invariant circle. By analogy with this theorem, one can be tempted to formulate the following conjecture:

\begin{conjecture} \label{cerclesursphere}
Any non-trivial action of the group $\mathrm{Homeo}_{0}(\mathbb{S}^{1})$ on the sphere (respectively on the closed disc)  is conjugate to one of the actions $\varphi_{K, \lambda}^{\mathbb{S}^{2}}$ (respectively to one of the actions $\varphi_{K, \lambda}^{\mathbb{D}^{2}}$).
\end{conjecture}

Notice that this theorem does not give directly a description of the conjugacy classes of such actions as two actions $\varphi_{K,\lambda}$ and $\varphi_{K',\lambda'}$ may be conjugate even though $K \neq K'$ or $\lambda \neq \lambda'$.
Now, let $K \subset [0,1]$ and $K' \subset [0,1]$ be two compact sets which contain the finite set $\left\{0,1\right\}$. Let $\lambda: [0,1]-K \rightarrow \left\{-1,+1 \right\}$ and $\lambda': [0,1]-K' \rightarrow \left\{-1,+1 \right\}$ be two maps which are constant on each connected component of their domain of definition. The following theorem characterizes when the actions $\varphi_{K,\lambda}$ and $\varphi_{K',\lambda'}$ are conjugate.

\begin{proposition}
The following statements are equivalent:
\begin{itemize}
\item the actions $\varphi_{K,\lambda}$ and $\varphi_{K',\lambda'}$ are conjugate;
\item either there exists an increasing homeomorphism $h:[0,1] \rightarrow [0,1]$ such that $h(K)=K'$ and such that $\lambda'\circ h=\lambda$ except on a finite number of connected component of $[0,1]-K$ or there exists a decreasing homeomorphism $h:[0,1] \rightarrow [0,1]$ such that $h(K)=K'$ and such that $\lambda'\circ h=-\lambda$ except on a finite number of connected component of $[0,1]-K$.
\end{itemize}
\end{proposition}

\begin{proof}
Let us begin by proving that the second statement implies the first one. 

If there exists an increasing homeomorphism $h$ which sends the compact subset $K \subset [0,1]$ to the compact subset $K' \subset [0,1]$, then the actions $\varphi_{K,\lambda}$ and $\varphi_{K', \lambda \circ h^{-1}}$ are conjugate. Indeed, denote by $\hat{h}$ the homeomorphism $[0,1] \rightarrow [0,1]$ which coincides with $h$ on $K$ and which, on each connected component $(r_{1},r_{2})$ of the complement of $K$, is the unique increasing linear homeomorphism $(r_{1},r_{2}) \rightarrow h((r_{1},r_{2}))$. Then the actions $\varphi_{K,\lambda}$ and $\varphi_{K', \lambda \circ h^{-1}}$ are conjugate via the homeomorphism
$$\begin{array}{rcl}
\mathbb{A} & \rightarrow & \mathbb{A} \\
(r, \theta) & \mapsto & (\hat{h}(r), \theta)
\end{array}
.$$
Similarly, suppose that $h$ is an orientation-reversing homeomorphism of the compact interval $[0,1]$ which sends the compact subset $K \subset [0,1]$ to the compact subset $K' \subset [0,1]$. As above, we denote by $\hat{h}$ the homeomorphism obtained from $h$ by sending linearly each connected component of the complement of $K$ to each connected component of the complement of $K'$. Then the actions $\varphi_{K,\lambda}$ and $\varphi_{K',-\lambda \circ h^{-1}}$ are conjugate via the homeomorphism
$$\begin{array}{rcl}
\mathbb{A} & \rightarrow & \mathbb{A} \\
(r, \theta) & \mapsto & (\hat{h}(r), \theta)
\end{array}
.$$

It suffices then to use the following lemma to complete the proof of the converse in the proposition.

\begin{lemma}
Let $K \subset [0,1]$ be a compact subset which contains the points $0$ and $1$. If $\lambda: [0,1]-K \rightarrow \left\{-1,+1 \right\}$ and $\lambda': [0,1]-K \rightarrow \left\{-1,+1 \right\}$ are continuous maps which are equal except on one connected component $(r_{1},r_{2})$ of the complement of $K$ where they differ, then the actions $\varphi_{K,\lambda}$ and $\varphi_{K,\lambda'}$ are conjugate. 
\end{lemma}

\begin{proof}
This lemma comes from the fact that the actions $a_{+}$ and $a_{-}$ are conjugate via the homeomorphism of the annulus $(r, \theta) \mapsto (r, \theta+r)$. More precisely, suppose that $\lambda((r_{1},r_{2}))=\left\{+1 \right\}$. Then the actions $\varphi_{K,\lambda}$ and $\varphi_{K,\lambda'}$ are conjugate via the homeomorphism $h$ defined as follows. The homeomorphism $h$ is equal to the identity on the Cartesian product of the complement of $(r_{1},r_{2})$ with the circle and is equal to $(r, \theta) \mapsto (r, \theta+ \frac{r-r_{1}}{r_{2}-r_{1}})$ on the open set $(r_{1},r_{2}) \times \mathbb{S}^{1}$.
\end{proof}

Let us establish now the other implication. Suppose that there exists a homeomorphism $g$ which conjugates the actions $\varphi_{K,\lambda}$ and $\varphi_{K',\lambda'}$. Now, for any angle $\alpha$, if we denote by $R_{\alpha}$ the rotation of angle $\alpha$:
$$ g \circ \varphi_{K,\lambda}(R_{\alpha})=\varphi_{K',\lambda'}(R_{\alpha})\circ g.$$
The homeomorphism $\varphi_{K,\lambda}(R_{\alpha})$ is the rotation of angle $\alpha$ of the annulus.
This means that, for any point $(r, \theta)$ of the closed annulus and any angle $\alpha \in \mathbb{S}^{1}$:
$$g((r, \theta+ \alpha))=g((r,\theta))+(0,\alpha).$$
In particular, the homeomorphism $g$ permutes the leaves of the foliation of the annulus by circles whose leaves are of the form $\left\{ r \right\} \times \mathbb{S}^{1}$.
Fix now a point $\theta_{0}$ on the circle. The homeomorphism $g$ sends the set $K \times \left\{ \theta_{0} \right\}$ of fixed points of $\varphi_{K,\lambda}(G_{\theta_{0}})$ (the points which are fixed by every element of this group) to the set $K' \times \left\{ \theta_{0} \right\}$ of fixed points of $\varphi_{K',\lambda'}(G_{\theta_{0}})$. From this and what is above, we deduce that for any $r \in K$ and any angle $\theta$, $g(r, \theta)=(h(r), \theta)$, where $h:K \rightarrow K'$ is a homeomorphism. Moreover, if the homeomorphism $g$ is orientation-preserving, then the homeomorphism $h$ is increasing and, if the homeomorphism $g$ is orientation-reversing, then the homeomorphism $h$ is decreasing. We can extend the homeomorphism $h$ to a homeomorphism of $[0,1]$ which sends $K$ to $K'$.
Notice that, for a connected component $(r_{1},r_{2})$ of the complement of $K$ in $[0,1]$, the homeomorphism $g$ sends the open set $(r_{1},r_{2})\times \mathbb{S}^{1}$ onto the open set $(r'_{1},r'_{2})\times \mathbb{S}^{1}$, where $(r'_{1},r'_{2})=h((r_{1},r_{2}))$ is a connected component of the complement of $K'$ in $[0,1]$.

It suffices now to establish that the condition on the maps $\lambda$ and $\lambda'$ is satisfied for the homeomorphism $h$. Suppose, to simplify the proof, that the homeomorphism $g$ is orientation-preserving and hence the homeomorphism $h$ is increasing: the case where the homeomorphism $g$ is orientation reversing can be treated similarly. Suppose by contradiction that there exists a sequence $((r_{1,n},r_{2,n}))_{n \in \mathbb{N}}$ of connected components of the open set $[0,1]-K$ such that:
\begin{itemize}
\item $\lambda((r_{1,n},r_{2,n}))=+1$;
\item $\lambda' ((r'_{1,n},r'_{2,n}))=-1$, where $(r'_{1,n},r'_{2,n})=h((r_{1,n},r_{2,n}))$;
\item the sequence $(r_{1,n})_{n \in \mathbb{N}}$ is monotonous and converges to a real number $r_{\infty}$.
\end{itemize}
We will prove then that either the curve $\left\{(r,\theta_{0}), r > r_{\infty} \right\}$ (if the sequence $(r_{1,n})_{n \in \mathbb{N}}$  is decreasing) or the curve $\left\{(r,\theta_{0}), r < r_{\infty} \right\}$ (if the sequence $(r_{1,n})_{n \in \mathbb{N}}$  is increasing) is sent by the homeomorphism $g$ to a curve which accumulates on $\left\{ r_{\infty} \right\} \times \mathbb{S}^{1}$, which is not possible. The hypothesis $\lambda((r_{1,n},r_{2,n}))=-1$ and $\lambda((r'_{1,n},r'_{2,n}))=+1$ for any $n$ would lead to the same contradiction.

To achieve this, it suffices to prove that, for any positive integer $n$, the homeomorphism $g$ sends the curve $\left\{ (r, \theta_{0}), r_{1,n}<r<r_{2,n} \right\}$ onto the curve $\left\{ (r, \theta_{0}+\frac{r-r'_{1,n}}{r'_{2,n}-r'_{1,n}}), r'_{1,n}<r<r'_{2,n} \right\}$. Fix a positive integer $n$. Observe that the restriction of the group of homeomorphisms $\varphi_{K,\lambda}(G_{\theta_{0}})$ to $(r_{1,n},r_{2,n}) \times \mathbb{S}^{1}$ has two invariant simple curves $\left\{ (r, \theta_{0}), r_{1,n}<r<r_{2,n} \right\}$ and $\left\{ (r, \theta_{0}-\frac{r-r_{1,n}}{r_{2,n}-r_{1,n}}), r_{1,n}<r<r_{2,n} \right\}$ on which the action is transitive. The action of this group is transitive on the two connected components of the complement of these two curves, which are open sets. Likewise, the restriction of the group of homeomorphisms $\varphi_{K',\lambda'}(G_{\theta_{0}})$ to $(r'_{1,n},r'_{2,n})$ has two invariant simple curves $\left\{ (r, \theta_{0}), r'_{1,n}<r<r'_{2,n} \right\}$ and $\left\{ (r, \theta_{0}+\frac{r-r'_{1,n}}{r'_{2,n}-r'_{1,n}}), r'_{1,n}<r<r'_{2,n} \right\}$ on which the action is transitive and the action of this group is transitive on the two connected components of the complement of these two curves, which are open sets. Therefore, the homeomorphism $g$ sends the curve $\left\{ (r, \theta_{0}), r_{1,n}<r<r_{2,n} \right\}$ to one of the curves $\left\{ (r, \theta_{0}), r'_{1,n}<r<r'_{2,n} \right\}$ or $\left\{ (r, \theta_{0}+\frac{r-r'_{1,n}}{r'_{2,n}-r'_{1,n}}), r'_{1,n}<r<r'_{2,n} \right\}$. Let us find now which of them is the image of the curve $\left\{ (r, \theta_{0}), r_{1,n}<r<r_{2,n} \right\}$. We fix the orientation of the circle induced by the orientation of the real line $\mathbb{R}$ and the covering map $\mathbb{R} \rightarrow \mathbb{S}^{1}=\mathbb{R}/\mathbb{Z}$. This orientation gives rise to an order on $\mathbb{S}^{1}-\left\{\theta_{0} \right\}$. Take a homeomorphism $f$ in $G_{\theta_{0}}$ different from the identity such that, for any $x \neq \theta_{0}$, $f(x) \geq x$. Then, for any $r \in (r_{1,n},r_{2,n})$, $p_{1} \circ \varphi_{K,\lambda}(f)(r,\theta_{0}) \geq r$, where $p_{1}: \mathbb{A}=[0,1] \times \mathbb{S}^{1} \rightarrow [0,1]$ is the projection, and the restriction of the homeomorphism $\varphi_{K,\lambda}(f)$ to the curve $\left\{ (r, \theta_{0}), r_{1,n}<r<r_{2,n} \right\}$ is different from the identity. Likewise, for any $r \in (r'_{1,n},r'_{2,n})$, $p_{1} \circ \varphi_{K',\lambda'}(f)(r,\theta_{0}) \leq r$ and $p_{1} \circ \varphi_{K',\lambda'}(f)(r,\theta_{0}+\frac{r-r'_{1,n}}{r'_{2,n}-r'_{1,n}}) \geq r$ and the restrictions of the homeomorphism $\varphi_{K',\lambda'}(f)$ to $\left\{ (r, \theta_{0}), r'_{1,n}<r<r'_{2,n} \right\}$ and to $\left\{ (r, \theta_{0}+\frac{r-r'_{1,n}}{r'_{2,n}-r'_{1,n}}), r'_{1,n}<r<r'_{2,n} \right\}$ are different from the identity. Moreover, the map
$$\begin{array}{rcl}
(r_{1,n},r_{2,n}) & \rightarrow & (r'_{1,n},r'_{2,n})\\
r & \mapsto & p_{1} \circ g(r, \theta_{0})
\end{array}$$
is strictly increasing as the homeomorphism $g$ was supposed to be orientation-preserving. This implies what we wanted to prove.
\end{proof}

\section{Continuous actions of $\mathrm{Homeo}_{c}(\mathbb{R})$ on the line}

Let us denote by $\mathrm{Homeo}_{c}(\mathbb{R})$ the group of compactly supported homeomorphisms of the real line $\mathbb{R}$. In this section, a proof of the following theorem is provided.

\begin{theorem} \label{actionsdersurr}
Let $\psi: \mathrm{Homeo}_{c}(\mathbb{R}) \rightarrow \mathrm{Homeo}(\mathbb{R})$ be a continuous group morphism whose image has no fixed point. Then there exists a homeomorphism $h$ of $\mathbb{R}$ such that, for any compactly supported homeomorphism $f$ of the real line:
$$\psi(f)=h \circ f \circ h^{-1}.$$
\end{theorem}

\underline{Remark}: This theorem is true without continuity hypothesis. However, we just need the above theorem in this article.

\underline{Remark}: This theorem also holds in the case of groups of diffeomorphisms, \emph{i.e.} any continuous action of the group of compactly-supported $C^{r}$ diffeomorphisms on the real line is topologically conjugate to the inclusion. The proof in this case is the same.

This theorem enables us to describe any continuous action of the group $\mathrm{Homeo}_{c}(\mathbb{R})$ on the real line as it suffices to consider the action on each connected component of the complement of the fixed point set of the action.

\begin{proof}
We fix a morphism $\psi$ as in the statement of the theorem. During this proof, for a subset $A \subset \mathbb{R}$, we will denote by $G_{A}$ the group of compactly supported homeomorphisms which pointwise fix a neighbourhood of $A$ and by $F_{A} \subset \mathbb{R}$ the closed set of fixed points of $\psi(G_{A})$ (the set of points which are fixed by every element of this group). Let us begin by giving a sketch of the proof of this theorem. We will first prove that, for a compact interval $I$ with non-empty interior, the set $F_{I}$ is compact and non-empty. As a consequence, for any real number $x$, the closed set $F_{x}$ is non-empty. Then we will prove that the sets $F_{x}$ are single-point sets. Let $\left\{ h(x) \right\} = F_{x}$. Then the map $h$ is a homeomorphism and satisfies the relation required by the theorem. Let us give details.

Notice that, as the group $\mathrm{Homeo}_{c}(\mathbb{R})$ is simple (see \cite{Fis}), the image of $\psi$ is contained in the group of increasing homeomorphisms of $\mathbb{R}$ (otherwise there would exist a nontrivial group morphism $\mathrm{Homeo}_{c}(\mathbb{R}) \rightarrow \mathbb{Z}/2\mathbb{Z}$) and the morphism $\psi$ is a one-to-one map.

\begin{lemma}
For a compact interval $I \subset \mathbb{R}$ with non-empty interior, the closed set $F_{I}$ is non-empty.
\end{lemma}

\begin{proof}
Fix a compact interval $I$ with non-empty interior. Take a non-zero vector field $X: \mathbb{R} \rightarrow \mathbb{R}$ supported in $I$.  The flow of this vector field defines a morphism 
$$ \begin{array}{rcl}
\mathbb{R} & \rightarrow & \mathrm{Homeo}_{c}(\mathbb{R}) \\
t & \mapsto & \varphi^{t}
\end{array}
.$$
Let us assume for the moment that the set $F$ of fixed points of the subgroup $\left\{ \psi(\varphi^{t}), t \in \mathbb{R} \right\}$ of $\mathrm{Homeo}(\mathbb{R})$ is non-empty. Notice that this set is not $\mathbb{R}$ as the morphism $\psi$ is one-to-one. As each homeomorphism $\varphi^{t}$ commutes with any element in $G_{I}$, we obtain that, for any element $g$ in $\psi(G_{I})$, $g(F)=F$. Moreover, as any element of the group $G_{I}$ can be joined to the identity by a continuous path which is contained in $G_{I}$ and as the morphism $\psi$ is continuous, any connected component $C$ of $F$ is invariant under $\psi(G_{I})$. The upper bounds and the lower bounds of these intervals which lie in $\mathbb{R}$ are then fixed points of the group $\psi(G_{I})$. This proves the lemma.

It remains to prove that the set $F$ is non-empty. Suppose by contradiction that it is empty. Then, for any real number $x$, the map
$$\begin{array}{rcl}
\mathbb{R} & \rightarrow & \mathbb{R} \\
t & \mapsto & \psi(\varphi^{t})(x)
\end{array}$$
is a homeomorphism. Indeed, if it was not onto, the supremum or the infimum of the image would provide a fixed point for $(\psi(\varphi^{t}))_{t \in \mathbb{R}}$. If it was not one-to-one, there would exist $t_{0} \neq 0$ such that $\psi(\varphi^{t_{0}})(x)=x$. Then, for any positive integer $n$, $\psi(\varphi^{t_{0}/2^{n}})(x)=x$ and, by continuity of $\psi$, for any real number $t$, $\psi(\varphi^{t})(x)=x$ and the point $x$ would be fixed by $(\psi(\varphi^{t}))_{t \in \mathbb{R}}$.  

Fix a real number $x_{0}$. Let $T_{x_{0}}: G_{I} \rightarrow \mathbb{R}$ be the map defined by
$$ \psi(f)(x_{0})= \psi(\varphi^{T_{x_{0}}(f)})(x_{0}).$$
The map $T_{x_{0}}$ is a group morphism as, for any homeomorphisms $f$ and $g$ in $G_{I}$,
$$\begin{array}{rcl}
\psi(\varphi^{T_{x_{0}}(f g)})(x_{0}) & = & \psi(fg)(x_{0}) \\
 & = & \psi(f) \psi(\varphi^{T_{x_{0}}(g)})(x_{0}) \\
 & = & \psi(\varphi^{T_{x_{0}}(g)}) \psi(f)(x_{0}) \\
 & = &  \psi(\varphi^{T_{x_{0}}(g)+T_{x_{0}}(f)})(x_{0}).
\end{array}
$$
However, the group $G_{I}$, which is isomorphic to the group $\mathrm{Homeo}_{c}(\mathbb{R}) \times \mathrm{Homeo}_{c}(\mathbb{R})$, is a perfect group: any element of this group can be written as a product of commutators. Therefore, the morphism $T_{x_{0}}$ is trivial. As the point $x_{0}$ is any point in $\mathbb{R}$, we deduce that the restriction of $\psi$ to $G_{I}$ is trivial, which is not possible as the morphism $\psi$ is one-to-one.
\end{proof}

\underline{Remark}: if $\psi(\mathrm{Homeo}_{c}(\mathbb{R})) \subset \mathrm{Homeo}_{c}(\mathbb{R})$, this lemma can be proved without continuity hypothesis. Indeed, consider an element $f$ in $\mathrm{Homeo}_{c}(\mathbb{R})$ supported in $I$. One of the connected components of the set of fixed points of $\psi(f)$ is of the form $(-\infty, a]$ for some $a$ in $\mathbb{R}$. This interval is necessarily invariant by the group $\psi(G_{I})$ which commutes with $f$. Hence, the point $a$ is a fixed point for the group $\psi(G_{I})$.

During the proof, we will often use the following elementary result:
\begin{lemma} \label{fragstab}
Let $I$ and $J$ be disjoint compact non-empty intervals. For any homeomorphism $g$ in $\mathrm{Homeo}_{c}(\mathbb{R})$, there exist homeomorphisms $g_{1} \in G_{I}$, $g_{2} \in G_{J}$ and $g_{3} \in G_{I}$ such that:
$$g=g_{1}g_{2}g_{3}.$$
\end{lemma}

\begin{proof}
Let $g$ be a homeomorphism in $\mathrm{Homeo}_{c}(\mathbb{R})$. Let $h_{1}$ be a homeomorphism in $G_{I}$ which sends the interval $g(I)$ to an interval which is in the same connected component of $\mathbb{R}-J$ as the interval $I$. Let $h_{2}$ be a homeomorphism in $G_{J}$ which is equal to $g^{-1}\circ h_{1}^{-1}$ on a neighbourhood of the interval $h_{1} \circ g(I)$. Then, the homeomorphism $h_{2} \circ h_{1} \circ g$ belongs to $G_{I}$. It suffices then to take $g_{1}=h_{1}^{-1}$, $g_{2}=h_{2}^{-1}$ and $g_{3}=h_{2} \circ h_{1} \circ g$ to conclude the proof of the lemma.
\end{proof}

Before stating the next lemma, observe that, for compact intervals $I$ with non-empty interiors, the sets $F_{I}$ are pairwise homeomorphic by an increasing homeomorphism. Indeed, let $I$ and $J$ be two such intervals. Then there exists a homeomorphism $\lambda$ in $\mathrm{Homeo}_{c}(\mathbb{R})$ such that $\lambda(I)=J$. Then $\lambda G_{I} \lambda^{-1}=G_{J}$. Taking the image by $\psi$, we obtain that $\psi(\lambda) \psi(G_{I}) \psi(\lambda)^{-1}=\psi(G_{J})$ and therefore $\psi(\lambda)(F_{I})=F_{J}$.

\begin{lemma}
For a compact interval $I \subset \mathbb{R}$ with non-empty interior, the closed set $F_{I}$ is compact.
\end{lemma}

\begin{proof}
Fix a compact interval $I$ with non-empty interior. Suppose by contradiction that there exists a sequence $(a_{k})_{k \in \mathbb{N}}$ of real numbers in $F_{I}$ which tends to $+\infty$ (if we suppose that it tends to $-\infty$, we obtain of course an analogous contradiction). Let us choose a compact interval $J$ which is disjoint from $I$. By the remark just before the lemma, there exists a sequence $(b_{k})_{k \in \mathbb{N}}$ of elements in $F_{J}$ which tends to $+\infty$. Take positive integers $n_{1}$, $n_{2}$ and $n_{3}$ such that $a_{n_{1}} < b_{n_{2}} < a_{n_{3}}$. Fix $x_{0}<a_{n_{1}}$. We notice then that for any homeomorphisms $g_{1} \in G_{I}$, $g_{2} \in G_{J}$ and $g_{3} \in G_{I}$, the following inequality is satisfied:
$$\psi(g_{1}) \psi(g_{2}) \psi(g_{3})(x_{0})<a_{n_{3}}.$$
Lemma \ref{fragstab} implies then that
$$ \overline{\left\{ \psi(g)(x_{0}), \ g \in \mathrm{Homeo}_{c}(\mathbb{R}) \right\}} \subset (-\infty,a_{n_{3}}].$$
The greatest element of the left-hand set is a fixed point of the image of $\psi$: this is not possible as this image was supposed to have no fixed point.
\end{proof}

\begin{lemma} \label{disjoint}
The closed sets $F_{x}$, where $x$ is a point of the real line, are non-empty, compact, pairwise disjoint and have an empty interior.
\end{lemma}

\begin{proof}
Notice that, if an interval $J$ is contained in an interval $I$, then $G_{I} \subset G_{J}$ and the same inclusion is true if we take the images of these subsets of $\mathrm{Homeo}_{c}(\mathbb{R})$ under the morphism $\psi$. That is why $F_{J} \subset F_{I}$. Now, for any finite family of intervals $(I_{n})_{n}$ whose intersection has non-empty interior, the intersection of the closed sets $F_{I_{n}}$ which contains the non-empty closed set $F_{\cap_{n} I_{n}}$, is not empty. Fix now a point $x$ in $\mathbb{R}$. Notice that
$$F_{x}= \bigcap_{I} F_{I},$$
where the intersection is taken over the compact intervals $I$ whose interior contains the point $x$. By compactness, this set is not empty and it is compact.

Take two points $x \neq y$ in $\mathbb{R}$ et let us prove now that the sets $F_{x}$ and $F_{y}$ are disjoint. If they were not disjoint, the group generated by $\psi(G_{x})$ and $\psi(G_{y})$ would have a fixed point $p$. However, by Lemma \ref{fragstab}, the groups $G_{x}$ and $G_{y}$ generate the group $\mathrm{Homeo}_{c}(\mathbb{R})$. Then, the point $p$ would be a fixed point of the group $\psi(\mathrm{Homeo}_{c}(\mathbb{R}))$, a contradiction.

Now, let us prove that the sets $F_{x}$ have an empty interior. Of course, given two points $x$ and $y$ of the real line, if $h$ is a homeomorphism in $\mathrm{Homeo}_{c}(\mathbb{R})$ which sends the point $x$ to the point $y$, then $\psi(h)(F_{x})=F_{y}$. Therefore, the sets $F_{x}$ are pairwise homeomorphic. If the sets $F_{x}$ had a non-empty interior, there would exist uncountably many pairwise disjoint open intervals of the real line, which is false. 
\end{proof}

\begin{lemma} \label{compl}
For any point $x_{0}$ of the real line, any connected component $C$ of the complement of $F_{x_{0}}$ meets one of the sets $F_{y}$, with $y \neq x$.
\end{lemma}

\begin{proof}
Let us fix a point $x_{0}$ in $\mathbb{R}$. Let $(a_{1},a_{2})$ be a connected component of the complement of $F_{x_{0}}$. It is possible that either $a_{1}= -\infty$ or $a_{2}=+ \infty$.

Let us prove by contradiction that there exists $y_{0} \neq x_{0}$ such that $F_{y_{0}} \cap (a_{1},a_{2}) \neq \emptyset$. Suppose that, for any point $y \neq x_{0}$, $F_{y} \cap (a_{1},a_{2})= \emptyset$. For any couple of real numbers $(z_{1},z_{2})$, choose a homeomorphism $h_{z_{1},z_{2}}$ in $\mathrm{Homeo}_{c}(\mathbb{R})$ such that $h_{z_{1},z_{2}}(z_{1})=z_{2}$. We claim that the open sets $\psi(h_{x_{0},y})((a_{1},a_{2}))$, for $y \in \mathbb{R}$, are pairwise disjoint, which is not possible as there would be uncountably many pairwise disjoint open intervals in $\mathbb{R}$. Indeed, if this was not the case, suppose that $\psi(h_{x_{0},y_{1}})((a_{1},a_{2})) \cap \psi(h_{x_{0},y_{2}})((a_{1},a_{2})) \neq \emptyset$ for $y_{1} \neq y_{2}$. As the union of the closed sets $F_{y}$ is invariant under the action $\psi$, then, for $i=1,2$, when $a_{i}$ is finite, $\psi(h_{x_{0},y_{1}})^{-1} \circ \psi(h_{x_{0},y_{2}})(a_{i}) \notin (a_{1},a_{2})$, and $\psi(h_{x_{0},y_{2}})^{-1} \circ \psi(h_{x_{0},y_{1}})(a_{i}) \notin (a_{1},a_{2})$ so that $\psi(h_{x_{0},y_{1}})(a_{i})=\psi(h_{x_{0},y_{2}})(a_{i})$. But this last equality cannot hold as the point on the left-hand side belongs to $F_{y_{1}}$ and the point on the right-hand side belongs to $F_{y_{2}}$ and we observed that these two closed sets were disjoint.
\end{proof}

\begin{lemma}
Each set $F_{x}$ contains only one point.
\end{lemma}

\begin{proof}
Suppose that there exists a real number $x$ such that the set $F_{x}$ contains two points $p_{1}<p_{2}$. By Lemma \ref{compl}, there exists a real number $y$ different from $x$ such that the set $F_{y}$ has a common point with the open interval $(p_{1},p_{2})$. Take a point $r<p_{1}$. Then, for any homeomorphisms $g_{1}$ in $G_{x}$, $g_{2}$ in $G_{y}$ and $g_{3}$ in $G_{x}$,
$$ \psi(g_{1}) \circ \psi(g_{2}) \circ \psi(g_{3})(r) < p_{2}.$$
By Lemma \ref{fragstab}, this implies that the following inclusion holds:
$$ \left\{ \psi(g)(r), g \in \mathrm{Homeo}_{c}(\mathbb{R}) \right\} \subset (-\infty, p_{2}].$$
The supremum of the left-hand set provides a fixed point for the action $\psi$, a contradiction.
\end{proof}

Take a homeomorphism $f$ in $\mathrm{Homeo}_{c}(\mathbb{R})$ and a point $x$ in $\mathbb{R}$. Then the homeomorphism $\psi(f)$ sends the only point $h(x)$ in $F_{x}$ to the only point $h(f(x))$ in $F_{f(x)}$. This implies that $\psi(f)\circ h= h \circ f$. Thus, it suffices to use Lemma \ref{homeomorphism} below to complete the proof of Theorem \ref{actionsdersurr}.
\end{proof}

\begin{lemma} \label{homeomorphism}
Let us denote by $h(x)$ the only point in the set $F_{x}$. Then the map $h$ is a homeomorphism. 
\end{lemma}

\begin{proof}
By Lemma \ref{disjoint}, the map $h$ is one-to-one.

Fix $x_{o} \in \mathbb{R}$ and let us prove that the map $h$ is continuous at $x_{0}$. Take a compactly supported $C^{1}$ vector field $\mathbb{R} \rightarrow \mathbb{R}$ which does not vanish on a neighbourhood of $x_{0}$ and denote by $(\varphi^{t})_{t \in \mathbb{R}}$ the flow of this vector field. Then, for any time $t$, $h(\varphi^{t}(x_{0}))= \psi(\varphi^{t})(h(x_{0}))$, which proves that the map $h$ is continuous at $x_{0}$.

Finally, let us prove that the map $h$ is onto. Notice that the interval $h(\mathbb{R})$ is invariant under the action $\psi$. Hence, if $\sup(h(\mathbb{R})) < + \infty$ (respectively $\inf(h(\mathbb{R})) > - \infty$), then the point $\sup(h(\mathbb{R}))$ (respectively $\inf(h(\mathbb{R}))$) would be a fixed point of the action $\psi$, a contradiction.
\end{proof}

\begin{proposition} \label{actionsdersurs}
Any action of the group $\mathrm{Homeo}_{c}(\mathbb{R})$ on the circle has a fixed point.
\end{proposition}

Hence, the description of the actions of the group $\mathrm{Homeo}_{c}(\mathbb{R})$ on the circle is given by an action of the group $\mathrm{Homeo}_{c}(\mathbb{R})$ on the real line which is homeomorphic to the circle minus one point. 

\begin{proof}
As $H^{2}_{b}(\mathrm{Homeo}_{c}(\mathbb{R}), \mathbb{Z})= \left\{ 0 \right\}$ (see \cite{Math}), the bounded Euler class of this action necessarily vanishes so that this action admits a fixed point (see \cite{Ghy} about the bounded Euler class of an action on the circle and its properties).
\end{proof}

Finally, let us recall a result which is proved in \cite{Mat}.

\begin{proposition} \label{actionsdessurs}
Any non-trivial action of the group $\mathrm{Homeo}_{0}(\mathbb{S}^{1})$ on the circle is a conjugacy by a homeomorphism of the circle.
\end{proposition}

\section{Actions on the annulus}

This section is devoted to the proof of Theorem \ref{cerclesursurfaces} in the case of morphisms with values in the group of homeomorphisms of the annulus. Fix a morphism $\varphi: \mathrm{Homeo}_{0}(\mathbb{S}^{1}) \rightarrow \mathrm{Homeo}(\mathbb{A})$. Recall that, by Theorem 4 and Proposition 2 in \cite{RS}, such a group morphism is necessarily continuous.

First, as the group $\mathrm{Homeo}_{0}(\mathbb{S}^{1})$ is simple (see \cite{Fis}), such a group morphism is either trivial or one-to-one. We assume in the rest of this section that this group morphism is one-to-one. Recall that the induced group morphism $\mathrm{Homeo}_{0}(\mathbb{S}^{1}) \rightarrow \mathrm{Homeo}(\mathbb{A})/\mathrm{Homeo}_{0}(\mathbb{A})$ is not one-to-one as the source is uncountable and the target is countable. Therefore, it is trivial and the image of the morphism $\varphi$ is contained in the group $\mathrm{Homeo}_{0}(\mathbb{A})$.

Now, recall that such  the subgroup of rotations of the circle is continuously isomorphic to the topological group $\mathbb{S}^{1}$. The image of this subgroup under the morphism $\varphi$ is a compact subgroup of the group of homeomorphisms of the closed annulus which is continuously isomorphic to $\mathbb{S}^{1}$. It is known that such a subgroup is conjugate to the rotation subgroup $\left\{ (r, \theta) \mapsto (r, \theta+\alpha), \alpha \in \mathbb{S}^{1} \right\}$ of the group of homeomorphisms of the annulus (see \cite{Kol}). This is the only place in this proof where we really need the continuity hypothesis. After possibly conjugating $\varphi$, we may suppose from now on that, for any angle $\alpha$, the morphism $\varphi$ sends the rotation of angle $\alpha$ of the circle to the rotation of angle $\alpha$ of the annulus.

\subsection{An invariant lamination}

Fix a point $\theta_{0}$ on the circle. Recall that the group $G_{\theta_{0}}$ is the group of homeomorphisms of the circle which fix a neighbourhood of the point $\theta_{0}$. Let us denote by $F_{\theta_{0}} \subset \mathbb{A}$ the closed subset of fixed points of $\varphi(G_{\theta_{0}})$ (\emph{i.e.} points which are fixed under every homeomorphism in this group). The action $\varphi$ induces actions on the two boundary components of the annulus which are circles. The action of $\varphi(G_{\theta_{0}})$ on each of these boundary components admits a fixed point by Proposition \ref{actionsdessurs}. Therefore, the set $F_{\theta_{0}}$ is non-empty. For any angle $\theta$, let $\alpha=\theta-\theta_{0}$. Then $G_{\theta}=R_{\alpha}G_{\theta_{0}}R_{\alpha}^{-1}$, where $R_{\alpha}$ denotes the rotation of the circle of angle $\alpha$. Therefore, $\varphi(G_{\theta})=R_{\alpha}\varphi(G_{\theta_{0}})R_{\alpha}^{-1}$, where $R_{\alpha}$ denotes here by abuse the rotation of angle $\alpha$ of the annulus, and $F_{\theta}=R_{\alpha}(F_{\theta_{0}})$. Let
$$B= \bigcup_{\theta \in \mathbb{S}^{1}}F_{\theta}=\bigcup_{\alpha \in \mathbb{S}^{1}}R_{\alpha}(F_{\theta_{0}}).$$
This set is of the form $B= K \times \mathbb{S}^{1}$, where $K$ is the image of $F_{\theta_{0}}$ under the projection $\mathbb{A}=[0,1] \times \mathbb{S}^{1} \rightarrow [0,1]$. The set $K$ is compact and contains the points $0$ and $1$. Let $\mathcal{F}$ be the foliation of the set $B$ whose leaves are the circles of the form $\left\{r \right\} \times \mathbb{S}^{1}$, where the real number $r$ belongs to the compact set $K$.

\begin{lemma} \label{fol}
Each leaf of the lamination $\mathcal{F}$ is preserved by the action $\varphi$.
\end{lemma}

\begin{proof}
Fix an angle $\theta$. Let us prove that, for any point $x$ in $F_{\theta}$ the orbit of the point $x$ under the action $\varphi$ is contained in $\left\{ R_{\alpha}(x), \alpha \in \mathbb{S}^{1} \right\} \subset \mathbb{A}$.
First, for any orientation-preserving homeomorphism $f$ of the circle which fixes the point $\theta$, the homeomorphism $\varphi(f)$ pointwise fixes the set $F_{\theta}$. Indeed, such a homeomorphism of the circle is the uniform limit of homeomorphisms which pointwise fix a neighbourhood of $\theta$ and the claim results from the continuity of the action $\varphi$.
Any orientation-preserving homeomorphism of the circle $g$ can be written $g=R_{\beta}f$, where $f$ is a homeomorphism which fixes the point $\theta$. Now, any point $x$ in $F_{\theta}$ is fixed under $\varphi(f)$ and and sent to a point in $\left\{ R_{\alpha}(x), \alpha \in \mathbb{S}^{1} \right\} \subset \mathbb{A}$ under the rotation $R_{\beta}=\varphi(R_{\beta})$, which proves the lemma.
\end{proof}

\begin{lemma} \label{fol2}
Each closed set $F_{\theta}$ intersects each leaf of the lamination $\mathcal{F}$ in exactly one point. Moreover, the map $h$ which, to any point $(r, \theta)$ of $K \times \mathbb{S}^{1} \subset \mathbb{A}$, associates the only point of $F_{\theta}$ on the leaf $\left\{ r \right\} \times \mathbb{S}^{1}$ is a homeomorphism of $K \times \mathbb{S}^{1}$. This homeomorphism conjugates the restriction of the action $\varphi_{K,\lambda}$ to $K \times \mathbb{S}^{1}$ to the restriction of the action $\varphi$ to $K \times \mathbb{S}^{1}$ for any continuous map $\lambda: [0,1]-K \rightarrow \left\{ -1,+1 \right\}$. Moreover, this homeomorphism is of the form $(r,\theta) \mapsto (r,\eta(r)+\theta)$ where $\eta: K \rightarrow \mathbb{S}^{1}$ is a continuous function.
\end{lemma}

\begin{proof}
By Lemma \ref{fol}, the action $\varphi$ preserves each set of the form $\left\{r \right\}\times \mathbb{S}^{1}$ and induces on each of these set an action on a circle. By Proposition \ref{actionsdessurs}, such actions are well-understood and the restriction of this action to a subgroup of the form $G_{\theta}$ admits exactly one fixed point (this action is non-trivial as the rotation subgroup acts non-trivially). This implies the first statement of the lemma.

 Take a sequence $(r_{n})_{n \in \mathbb{N}}$ of elements in $K$ which converges to a real number $r$ in $K$. Then, as the point $h(r_{n}, \theta_{0})$ belongs to $\left\{ r_{n} \right\} \times \mathbb{S}^{1}$ and to $F_{\theta_{0}}$, the accumulation points of the sequence $(h(r_{n},\theta_{0}))_{n}$ belong to $\left\{ r \right\} \times \mathbb{S}^{1}$ and to $F_{\theta_{0}}$: there is only one accumulation point which is the point $h(r,\theta_{0})$. For any angle $\theta$, we have $h(r,\theta)=R_{\theta-\theta_{0}}\circ h(r,\theta_{0})$. This implies that the map $h$ is continuous. This map is one-to-one: if $h(r,\theta)=h(r',\theta')$, this last point belongs to $\left\{r \right\} \times \mathbb{S}^{1}=\left\{r' \right\} \times \mathbb{S}^{1}$ so that $r=r'$ and $R_{\theta-\theta_{0}}\circ h(r,\theta_{0})=R_{\theta'-\theta_{0}}\circ h(r,\theta_{0})$ so that $\theta=\theta'$. This map is onto by definition of the set $B=K \times \mathbb{S}^{1}$ which is the union of the $F_{\theta}$. As this map is defined on a compact set, it is a homeomorphism.

Take any orientation-preserving homeomorphism $f$ of the circle with $f(\theta)=\theta'$. Notice that $G_{\theta'}=f G_{\theta}f^{-1}$. Therefore $\varphi(G_{\theta'})=\varphi(f) \varphi(G_{\theta}) \varphi(f)^{-1}$ and $\varphi(f)(F_{\theta})=F_{\theta'}$. So, for any $(r,\theta) \in K \times \mathbb{S}^{1}$, as the action $\varphi$ preserves each leaf of the lamination $\mathcal{F}$, the homeomorphism $\varphi(f)$ sends the point $h(r,\theta)$, which is the unique point in the intersection $F_{\theta} \cap \left\{ r \right\} \times \mathbb{S}^{1}$, to the unique point in the intersection $F_{\theta'} \cap \left\{ r \right\} \times \mathbb{S}^{1}$, which is $h(r,\theta')$. This implies that
$$ \varphi(f)\circ h(r,\theta)= h \circ \varphi_{K,\lambda}(f)(r,\theta).$$
Now, denote by $\eta(r)$ the second projection of $h(r,0)$. As the set $F_{\theta}$ is the image by the rotation of angle $\theta$ of the set $F_{0}$, we have $h(r,\theta)=h(r,0)+(0,\theta)$ and $h(r,\theta)=(r,\eta(r)+\theta)$.
\end{proof}

\subsection{Action outside the lamination}

In this section, we study the action $\varphi$ on each connected component of $\mathbb{A}-K \times \mathbb{S}^{1}$. Let $A=[r_{1},r_{2}] \times \mathbb{S}^{1}$ be the closure of such a component. By the last subsection, such a set is invariant under the action $\varphi$. This subsection is dedicated to the proof of the following proposition.

\begin{proposition} \label{outside}
The restriction $\varphi^{A}$ of the action $\varphi$ to $A$ is conjugate to $a_{+}$ (or equivalently to $a_{-}$) via an orientation preserving homeomorphism.
\end{proposition}

\begin{proof}
 Notice that, for any point $\theta$ of the circle, the action $\varphi^{A}_{|G_{\theta}}$ admits no fixed point in the interior of $A$ by definition of $K \times \mathbb{S}^{1}$: we will often use this fact.

Let us begin by sketching the proof of this proposition. We prove that, for any point $\theta$ of the circle, the morphism $\varphi^{A}_{| G_{\theta}}$ can be lifted to a morphism $\tilde{\varphi}^{A}_{\theta}$ from the group $G_{\theta}$ to the group $\mathrm{Homeo}_{\mathbb{Z}}([r_{1},r_{2}] \times \mathbb{R})$ of homeomorphisms of the closed band which commute to the translations $(r,x) \mapsto (r,x+n)$, where $n$ is an integer. Moreover, this group morphism can be chosen so that it has a bounded orbit. We will then find a continuum $\widetilde{G_{\theta}}$ with empty interior which touches both boundary components of the band and is invariant under the action $\tilde{\varphi}^{A}_{\theta}$. Then we prove that the sets $G_{\theta}= \pi(\widetilde{G_{\theta}})$, where $\pi: [r_{1},r_{2}] \times \mathbb{R} \rightarrow A$ is the projection, are pairwise disjoint and are simple paths which join the two boundary components of the annulus $A$. We see that the group $\varphi(G_{\theta} \cap G_{\theta'})$, for $\theta \neq \theta'$, admits a unique fixed point $a(\theta,\theta')$ on $G_{\theta}$. This last map $a$ turns out to be continuous and allows us to build a conjugacy between $\varphi^{A}$ and $a_{+}$.

The following lemma is necessary to build the invariant sets $G_{\theta}$:

For a homeomorphism $g$ in $\mathrm{Homeo}_{0}(A)$, we denote by $\tilde{g}$ the lift of $g$ to $\mathrm{Homeo}_{\mathbb{Z}}([r_{1},r_{2}] \times \mathbb{R})$  (this means that $\pi \circ \tilde{g}= g \circ \pi$) with $\tilde{g}((r_{1},0)) \in [r_{1},r_{2}] \times [-1/2,1/2)$. Let us denote by $D \subset \mathbb{R}^{2}$ the fundamental domain $[r_{1},r_{2}] \times [-1/2,1/2]$ for the action of $\mathbb{Z}$ on $[r_{1},r_{2}] \times \mathbb{R}$.

\begin{lemma}\label{domfondann}
The map $\mathrm{Homeo}_{0}(\mathbb{S}^{1}) \rightarrow \mathbb{R}_{+}$, which associates, to any homeomorphism $f$ in $\mathrm{Homeo}_{0}(\mathbb{S}^{1})$, the diameter of the image under $\widetilde{\varphi(f)}$ (or equivalently under any lift of $\varphi(f)$) of the fundamental domain $D$, is bounded.
\end{lemma}

\underline{Remark}: the continuity hypothesis on the morphism $\varphi$ is not used in the proof of this lemma. 

\begin{proof}
For a group $G$ generated by a finite set $S$ and for an element $g$ in $G$, we denote by $l_{S}(g)$ the word length of $g$ with respect to $S$, which is the minimal number of factors necessary to write $g$ as a product of elements of $S \cup S^{-1}$, where $S^{-1}$ is the set of inverses of elements in $S$. In order to prove the lemma, we need the following result which can be easily deduced from Lemma 4.4 in \cite{Mil}, which is inspired from \cite{Avi}.

\begin{lemma}
There exist constants $C>0$ and $C'\in \mathbb{R}$ such that, for any sequence $(f_{n})_{n \in \mathbb{N}}$ of elements of the group $\mathrm{Homeo}_{0}(\mathbb{S}^{1})$, there exists a finite set $S \subset \mathrm{Homeo}_{0}(\mathbb{S}^{1})$ such that:
\begin{itemize}
\item for any integer $n$, the element $f_{n}$ belongs to the group generated by $S$;
\item $\forall n \in \mathbb{N}, \ l_{S}(f_{n}) \leq C log(n)+C'$.
\end{itemize}
\end{lemma}

Let us prove now Lemma \ref{domfondann} by contradiction. Suppose that there exists a sequence $(f_{n})_{n \in \mathbb{N}}$ of elements of the group $\mathrm{Homeo}_{0}(\mathbb{S}^{1})$ such that, for any integer $n$:
$$ \mathrm{diam}(\widetilde{\varphi(f_{n})}(D)) \geq n.$$
We apply then the above lemma to this sequence to obtain a finite subset $S$ contained in $\mathrm{Homeo}_{0}(\mathbb{S}^{1})$ such that the conclusion of the lemma holds. Let $f_{n}=s_{1,n}  s_{2,n}  \ldots s_{w_{n},n}$, where $w_{n} \leq Clog(n)+C'$ and the $s_{i,j}$'s are elements of $S \cup S^{-1}$. We prove now that this implies that the diameter of $\widetilde{\varphi(f_{n})(D)}$ grows at most at a logarithmic speed, which is a contradiction with the hypothesis we made on the sequence $(f_{n})_{n}$. Let us denote by $M$ the maximum of the quantities $\left\| \widetilde{\varphi(s)}(x)-x \right\|$, where $s$ varies over $S \cup S^{-1}$ and the point $x$ varies over $\mathbb{R}^{2}$ (or equivalently over the compact set $D$) and where $\left\| . \right\|$ denotes the euclidean norm. Take now two points $x$ and $y$ in $D$. Then for any integer $n$
\small{$$\begin{array}{l}
\left\| \widetilde{\varphi(s_{1,n})} \widetilde{\varphi(s_{2,n})}  \ldots  \widetilde{\varphi(s_{w_{n},n})}(x)- \widetilde{\varphi(s_{1,n})} \widetilde{\varphi(s_{2,n})}  \ldots  \widetilde{\varphi(s_{w_{n},n})}(y) \right\| \leq \\
 \left\|\widetilde{\varphi(s_{1,n})}  \widetilde{\varphi(s_{2,n})}  \ldots  \widetilde{\varphi(s_{w_{n},n})}(x)-x \right\| + \left\|\widetilde{\varphi(s_{1,n})}  \widetilde{\varphi(s_{2,n})}  \ldots \widetilde{\varphi(s_{w_{n},n})}(y)-y \right\| + \left\|x-y\right\|.
\end{array}$$}
\normalsize{But for any point $z$ in $D$, we have:}
\small{$$\begin{array}{rcl}
\left\|\widetilde{\varphi(s_{1,n})}  \widetilde{\varphi(s_{2,n})}  \ldots  \widetilde{\varphi(s_{w_{n},n})}(z)-z \right\| & \leq & \sum \limits_{k=1}^{w_{n}-1} \left\|\widetilde{\varphi(s_{k,n})} \ldots \widetilde{\varphi(s_{w_{n},n})}(z)-\widetilde{\varphi(s_{k+1,n})} \ldots \widetilde{\varphi(s_{w_{n},n})}(z) \right\| \\
 & \leq & (w_{n}-1)M.
\end{array}
$$}
\normalsize{Hence 
$$\mathrm{diam}(\widetilde{\varphi(s_{1,n})}  \widetilde{\varphi(s_{2,n})}  \ldots  \widetilde{\varphi(s_{w_{n},n})}(D))=\mathrm{diam}(\widetilde{\varphi(f_{n})}(D)) \leq 2(w_{n}-1)M+\mathrm{diam}(D),$$ which is in contradiction with the hypothesis we made on the sequence $(f_{n})_{n}$.}
\end{proof}

Let $\theta_{0}$ be a point of the circle.

\begin{lemma} \label{lift}
There exists a group morphism $\tilde{\varphi}^{A}_{\theta_{0}} : G_{\theta_{0}} \rightarrow \mathrm{Homeo}_{\mathbb{Z}}([r_{1},r_{2}] \times \mathbb{R})$ such that:
\begin{itemize}
\item for any homeomorphism $f$ in $G_{\theta_{0}}$, $\Pi \circ \tilde{\varphi}^{A}_{\theta_{0}}(f)= \varphi^{A}(f)$, where $\Pi:\mathrm{Homeo}_{\mathbb{Z}}([r_{1},r_{2}] \times \mathbb{R}) \rightarrow \mathrm{Homeo}_{0}(A)$ is the projection;
\item the subset $\left\{\tilde{\varphi}^{A}_{\theta_{0}}(f)((r_{1},0)), f \in G_{\theta_{0}} \right\}$ of the band $[r_{1},r_{2}] \times \mathbb{R}$ is bounded.
\end{itemize}
Moreover, the morphism $\tilde{\varphi}^{A}_{\theta_{0}}$ is continuous.
\end{lemma}

\underline{Remark}: The continuity hypothesis on the morphism $\varphi$ is not necessary for the first part of this lemma. However, we will use it to simplify the proof.

\begin{proof}
 As the topological space $G_{\theta_{0}}$ is contractible and the map $\Pi:\mathrm{Homeo}_{\mathbb{Z}}([r_{1},r_{2}] \times \mathbb{R}) \rightarrow \mathrm{Homeo}_{0}(A)$ is a covering, there exists a (unique) continuous map $\eta: G_{\theta_{0}} \rightarrow \mathrm{Homeo}_{\mathbb{Z}}([r_{1},r_{2}] \times \mathbb{R})$ which lifts the map $\varphi^{A}_{|G_{\theta_{0}}}$ and sends the identity to the identity. Then the map
$$\begin{array}{rcl}
G_{\theta_{0}} \times G_{\theta_{0}} & \rightarrow &  \mathrm{Homeo}_{\mathbb{Z}}([r_{1},r_{2}] \times \mathbb{R})\\
(f,g) & \mapsto & \eta(fg)^{-1} \eta(f) \eta(g)
\end{array}
$$
is continuous and its image is contained in the discrete space of integral translations: it is constant and the map $\eta$ is a group morphism. Two group morphisms which lift the group morphism $\varphi^{A}_{|G_{\theta_{0}}}$ differ by a morphism $G_{\theta_{0}} \rightarrow \mathbb{Z}$. However, as the group $G_{\theta_{0}}$ is simple (hence perfect), such a group morphism is trivial and $\eta= \tilde{ \varphi}^{A}_{\theta_{0}}$. The action $\varphi^{A}_{|G_{\theta_{0}}}$ admits fixed points on the boundary of the annulus $A$.  Hence, as the space $G_{\theta_{0}}$ is path-connected and $\eta(Id)=Id$, the action $\eta$ admits fixed points on the boundary of $[r_{1},r_{2}] \times \mathbb{R}$ and any orbit on one of these boundary components is bounded.
\end{proof}

Let
$$F= \overline{ \bigcup_{f \in G_{x_{0}}} \tilde{\varphi}^{A}_{\theta_{0}}(f)([r_{1},r_{2}] \times(-\infty,0])}.$$
By the two above lemmas, there exists $M>0$ such that $F \subset [r_{1},r_{2}] \times (-\infty,M]$. Moreover, the closed set $F$ is invariant under the action $\tilde{\varphi}^{A}_{\theta_{0}}$. Denote by $U$ the connected component of the complement of $F \cup \left\{r_{1},r_{2} \right\} \times \mathbb{R}$ which contains the open subset $(r_{1},r_{2}) \times (M,+\infty)$. By construction, the open set $U$ is invariant under the action $\tilde{\varphi}^{A}_{\theta_{0}}$ (the interior of a fundamental domain far on the right must be sent in $U$ by any homeomorphism in the image of $\tilde{\varphi}^{A}_{\theta_{0}}$, by the two above lemmas). 
Consider the topological space $B$ which is the disjoint union of the band $[r_{1},r_{2}] \times \mathbb{R}$ with a point $\left\{ + \infty \right\}$ and for which a basis of neighbourhood of the point $+\infty$ are given by the sets of the form $[r_{1},r_{2}] \times (A,+\infty) \cup \left\{ + \infty \right\}$. Now, let us consider the prime end compactification of $U \subset B$ (see \cite{Math} for the background on prime end compactification). The space of prime ends of the simply connected open set $U$, on which there is a natural action $\psi$ of the group $G_{\theta_{0}}$ induced by the action $\tilde{ \varphi}^{A}_{\theta_{0}}$,  is homeomorphic to a circle. By the following lemma, the action $\psi$ is continuous.

\begin{lemma}
Let $W$ be a simply connected relatively compact open subset of the plane. Let us denote by $B(W)$ the space of prime ends of $W$. The map 
$$t: \mathrm{Homeo}(\overline{W}) \rightarrow \mathrm{Homeo}(B(W)),$$
which, to a homeomorphism $f$ of $\overline{W}$, associates the induced homeomorphism on the space of prime ends of $W$, is continuous.
\end{lemma}

\begin{proof}
As the map $t$ is a group morphism and as the space $B(W)$ is homeomorphic to the circle, it suffices to prove that, for any prime end $\xi$ in $B(W)$, the map
$$\begin{array}{rcl}
\mathrm{Homeo}(\overline{W}) & \rightarrow & B(W) \\
f & \mapsto & t(f)(\xi)
\end{array}$$
is continuous at the identity. Let us fix such a prime end $\xi$. Let us denote by 
$$V_{1}\supset V_{2} \supset \ldots \supset V_{n} \supset \ldots$$
a prime chain which defines the prime end $\xi$. If we denote by $\tilde{V}_{n}$ the space of prime points of $W$ which divide $V_{n}$, then the $\tilde{V}_{n}$'s are a basis of neighbourhoods of the point $\xi$ (see Section 3 in \cite{Math}). Let us fix a natural integer $n$ and $p>n$. If the uniform distance between a homeomorphism $f$ in $\mathrm{Homeo}(\overline{W})$ and the identity  is lower than the distance between the frontier $\mathrm{Fr}_{W}(V_{n})$ of $V_{n}$ in $W$ and the frontier $\mathrm{Fr}_{W}(V_{p})$ of $V_{p}$ in $W$, then the set $f(\mathrm{Fr}_{W}(V_{n}))$ does not meet $\mathrm{Fr}_{W}(V_{p})$. By Lemma 4 in \cite{Math}, if the homeomorphism $f$ is sufficiently close to the identity, then $f(V_{p}) \subset V_{n}$ and $t(f)(\tilde{V}_{p}) \subset \tilde{V}_{n}$. This implies that, for $f$ in such a neighbourhood, the point $t(f)(\xi)$ belongs to $\tilde{V}_{n}$, which is what we wanted to prove.
\end{proof}

Take a prime end $\xi$ of $U$. The \emph{principal set} of $\xi$ is the set of points $p$ in $B$, called principal points of $\xi$, such that there exists a prime chain
$$V_{1}\supset V_{2} \supset \ldots \supset V_{n} \supset \ldots$$
defining $\xi$ such that the sequence of frontiers in $U$ of $V_{n}$ converges for the Hausdorff topology to the single-point set $\left\{ p \right\}$. This set is compact and connected.
Consider the subset of prime ends of $U$ whose principal set contains a point of $\left\{r_{1},r_{2} \right\} \times \mathbb{R} \cup \left\{ + \infty \right\}$. This set is invariant under $\psi$. Denote by $I$ a connected component of the complement of this set (the complement of this set is non empty by \cite{Math} because there exists an arc $[0,+\infty) \rightarrow U$ which converges as $t$ tends to $+\infty$ to a point which belongs to the frontier of $U$ but does not belong to $\left\{r_{1},r_{2} \right\} \times \mathbb{R} \cup \left\{ + \infty \right\}$). By path-connectedness of the group $G_{\theta_{0}}$ and continuity of the action $\psi$, the open interval $I$ is invariant under the action $\psi$. Let $\psi'$ be the restriction of the action $\psi$ to the interval $I$.

\begin{lemma} \label{transprime}
The action $\psi'$ has no fixed point. Therefore, the interval $I$ is open.
\end{lemma}

\begin{proof}
Suppose by contradiction that the action $\psi'$ admits a fixed point $\xi$. Then, the principal set of the prime end $\xi$ would provide a compact subset of the annulus $(r_{1},r_{2}) \times \mathbb{R}$ which is invariant under the action $\tilde{\varphi}^{A}_{\theta_{0}}$, a contradiction with Lemma \ref{invariantcompact} below.

\end{proof}

\begin{lemma} \label{invariantcompact}
The action $\tilde{\varphi}^{A}_{\theta_{0}}$ admits no nonempty compact connected invariant set contained in $(r_{1},r_{2}) \times \mathbb{R}$.
\end{lemma}

\begin{proof}
Consider the unbounded component of the complement of this set. The action $\tilde{\varphi}^{A}_{\theta_{0}}$ induces an action $\eta$ on the space of prime ends of this set, which is homeomorphic to a circle. By Proposition \ref{actionsdersurs}, such an action $\eta$ has a fixed point.

If the set of fixed points of this action has non-empty interior, as accessible prime ends are dense in this set (see \cite{Math}), there exists an accessible prime end which is fixed under $\eta$. Therefore, the only point in the principal set of this prime end, which is contained in the interior of the band $[r_{1},r_{2}] \times \mathbb{R}$, is fixed under the action $\tilde{\varphi}^{A}_{\theta_{0}}$.

Suppose now that this set of fixed points has an empty interior. Take a connected component of the complement of the set of fixed points of $\eta$ and an endpoint $e$ of this interval.  Take a closed interval $J$ of the circle whose interior contains the point $\theta_{0}$. Denote by $G_{J}$ the subgroup of $\mathrm{Homeo}_{0}(\mathbb{S}^{1})$ consisting of homeomorphisms which pointwise fix a neighbourhood of $J$. Then, according to Section 3 in which we describe the continuous actions of the group $\mathrm{Homeo}_{c}(\mathbb{R})$ on the real line, the set of fixed points of $\psi(G_{J})$ contains a closed interval $\mathcal{J}$ with non-empty interior which contains the point $e$ (not necessarily in its interior). As accessible prime ends are dense in $P$, this closed interval contains an accessible prime end. Then, as the principal set of this prime end is reduced to a point $p$, this point $p$ is fixed by the group $\tilde{\varphi}^{A}_{\theta_{0}}(G_{J})$. As a result, the group $\tilde{\varphi}^{A}_{\theta_{0}}(G_{J})$ of homeomorphisms of the band $[r_{1},r_{2}] \times \mathbb{R}$ admits a non-empty set $H_{J}$ of fixed points which is contained in the closure $C_{J}$ of the union of the principal sets of prime ends in $\mathcal{J}$, which is compact. Moreover, this set is contained in the interior of the band $[r_{1},r_{2}] \times \mathbb{R}$. For any closed interval $J'$ whose interior contains the point $\theta_{0}$ and which is contained in $J$, the set $H_{J'}$ of fixed points of $\tilde{\varphi}^{A}_{\theta_{0}}(G_{J'})$ which are contained in $C_{J}$ is non empty. Moreover, if the interval $J''$ is contained in the interval $J'$, then $H_{J''} \subset H_{J'}$. By compactness, the intersection of those sets, which is contained in the set of fixed point of the action $\tilde{\varphi}^{A}_{\theta_{0}}$, is non-empty.
\end{proof}

By the description of continuous actions of the group $\mathrm{Homeo}_{c}(\mathbb{R})$ on the real line with no global fixed points (see Section 3 of this article), the action $\psi$ is transitive on the open interval $I$. Hence, all the prime ends in $I$ are accessible, by density of accessible prime ends. Moreover, by this same description, for any point $\theta \neq \theta_{0}$ of the circle, there is a unique prime end $e_{\theta} \in I$ which is fixed under $\psi_{|G_{\theta_{0}} \cap G_{\theta}}$ and the interval $I$ is the union of these prime ends. For any point $\theta \neq \theta_{0}$ of the circle, let us denote by $\tilde{x}_{\theta}$ the unique point in the principal set of the prime end $e_{\theta}$. For any homeomorphism $f$ in $G_{\theta_{0}}$ which sends the point $\theta \neq \theta_{0}$ to the point $\theta'$, then $\tilde{\varphi}^{A}_{\theta_{0}}(f)(\tilde{x}_{\theta})=\tilde{x}_{\theta'}$, as $G_{\theta'}=f G_{\theta}f^{-1}$. Let $\widetilde{L_{\theta_{0}}}= \left\{ \tilde{x}_{\theta}, \theta \in \mathbb{S}^{1}- \left\{ \theta_{0} \right\} \right\}$ and $L_{\theta_{0}}=\pi(\widetilde{L_{\theta_{0}}})$.

\begin{lemma} \label{continuity}
The map
$$ \begin{array}{rcl}
\mathbb{S}^{1}-\left\{ \theta_{0} \right\} & \rightarrow & (r_{1},r_{2}) \times \mathbb{R} \\
 \theta & \mapsto & \tilde{x}_{\theta}
\end{array}$$
is one-to-one and continuous. Moreover, the limit sets $\lim_{\theta \rightarrow \theta_{0}^{+}} \tilde{x}_{\theta}$ and $\lim_{\theta \rightarrow \theta_{0}^{-}} \tilde{x}_{\theta}$ each contain exactly one point of the boundary $\left\{r_{1},r_{2} \right\} \times \mathbb{R}$ of the band.
\end{lemma}

\begin{proof}
Take $\theta' \neq \theta$ and suppose by contradiction that $\tilde{x}_{\theta}=\tilde{x}_{\theta'}$. Take $\theta'' \neq \theta$ in the same connected component of $\mathbb{S}^{1}- \left\{ \theta_{0}, \theta \right\}$ as $\theta'$ such that $\tilde{x}_{\theta''} \neq \tilde{x}_{\theta}$. Consider a homeomorphism $f$ of the circle in $G_{\theta_{0}} \cap G_{\theta}$ which sends the point $\theta'$ to the point $\theta''$. Then the homeomorphism $\tilde{\varphi}^{A}_{\theta_{0}}(f)$ of the band fixes the point $\tilde{x}_{\theta}$ and sends the point $\tilde{x}_{\theta'}$ to the point $\tilde{x}_{\theta''}$, which is not possible. The map is one-to-one.

Now, let us prove that the considered map is continuous. Take a smooth vector field of the circle which vanishes only on a small connected neighbourhood $N$ of $\theta_{0}$. Let us denote by $(h^{t})_{t \in \mathbb{R}}$ the one-parameter group generated by this vector field. Fix a point $\theta_{1} \in \mathbb{S}^{1}- N$. For any point $\theta \in \mathbb{S}^{1}- N$, denote by $t(\theta)$ the unique time $t$ such that $h^{t}(\theta_{1})=\theta$. The map $\theta \mapsto t(\theta)$ is then a homeomorphism $\mathbb{S}^{1}- N \rightarrow \mathbb{R}$. Now, the relation $\tilde{x}_{\theta}=\varphi(h^{t(\theta)})(\tilde{x}_{\theta_{1}})$ and the continuity of the action $\varphi$ implies that the map $\theta \mapsto \tilde{x}_{\theta}$ is continuous, as the neighbourhood $N$ can be taken arbitrarily small.

Lemma \ref{invariantcompact} implies that the intersection of each of the limit sets $\lim_{\theta \rightarrow \theta_{0}^{+}} \tilde{x}_{\theta}$ and $\lim_{\theta \rightarrow \theta_{0}^{-}} \tilde{x}_{\theta}$ with the boundary $\left\{r_{1},r_{2} \right\} \times \mathbb{R}$ of the band is nonempty. Indeed, otherwise, these limit sets would provide a nonempty compact connected invariant set for the action $\tilde{\varphi}^{A}_{\theta_{0}}$.

It remains to prove that the intersections $\lim_{\theta \rightarrow \theta_{0}^{+}} \tilde{x}_{\theta} \cap \left\{r_{1},r_{2} \right\} \times \mathbb{R}$ and $\lim_{\theta \rightarrow \theta_{0}^{-}} \tilde{x}_{\theta} \cap \left\{r_{1},r_{2} \right\} \times \mathbb{R}$ are reduced to a point. Suppose for instance that the intersection $\lim_{\theta \rightarrow \theta_{0}^{+}} \tilde{x}_{\theta} \cap \left\{r_{1} \right\} \times \mathbb{R}$ contains at least two points. Recall that, by Proposition \ref{actionsdessurs}, the restriction of the action $\varphi$ to $\left\{ r_{1} \right\} \times \mathbb{S}^{1}$ is a conjugacy by a homeomorphism $\mathbb{S}^{1} \rightarrow \left\{ r_{1} \right\} \times \mathbb{S}^{1}$. Therefore, the action $\varphi^{A}_{|G_{\theta_{0}}}$ fixes a point $p$ and is transitive on $\left\{ r_{1} \right\} \times \mathbb{S}^{1}- \left\{p \right\}$. Moreover, as any orbit of the action $\tilde{\varphi}^{A}_{\theta_{0}}$ is bounded by Lemmas \ref{domfondann} and \ref{lift}, this last action pointwise fixes the set $\pi^{-1}(\left\{ p \right\})$ and is transitive on each connected component of $\left\{ r_{1} \right\} \times \mathbb{R}-\pi^{-1}(\left\{ p \right\})$. Therefore, the intersection $\lim_{\theta \rightarrow \theta_{0}^{+}} \tilde{x}_{\theta} \cap \left\{r_{1} \right\} \times \mathbb{R}$, which is closed and invariant under the action $\tilde{\varphi}^{A}_{\theta_{0}}$, contains two distinct lifts $\tilde{p}$ and $\tilde{p}'$ of the point $p$. Take a sequence $(\tilde{x}_{\theta_{n}})_{n \in \mathbb{N}}$ of points in $\widetilde{L_{\theta_{0}}}$, where $\theta_{n}$ tends to $\theta_{0}^{+}$, which converges to the point $\tilde{p}$ and a sequence $(\tilde{x}_{\theta'_{n}})_{n \in \mathbb{N}}$ of points in $\widetilde{L_{\theta_{0}}}$ , where $\theta'_{n}$ tends to $\theta_{0}^{+}$, which converges to the point $\tilde{p}'$. Taking a subsequence if necessary, we may suppose that the sequences $(\theta_{n})_{n}$ and $(\theta'_{n})_{n}$ are monotonous and that, for any integer $n$, the angle $\theta'_{n}$ is between the angles $\theta_{n}$ and $\theta_{n+1}$. Take a homeomorphism $f$ in $G_{\theta_{0}}$ which, for any $n$, sends the point $\theta_{n}$ to the point $\theta'_{n}$. Then, for any $n$, the homeomorphism $\tilde{\varphi}^{A}_{\theta_{0}}(f)$ sends the point $\tilde{x}_{\theta_{n}}$ to the point $\tilde{x}_{\theta'_{n}}$. By continuity, this homeomorphism sends the point $\tilde{p}$ to the point $\tilde{p}'$. This is not possible as these points are fixed under the action $\tilde{\varphi}^{A}_{\theta_{0}}$.
\end{proof}

Note that we do not know for the moment that the points $\lim_{\theta \rightarrow \theta_{0}^{+}} \tilde{x}_{\theta}$ and $\lim_{\theta \rightarrow \theta_{0}^{-}} \tilde{x}_{\theta}$ lie on different boundary components of $A$.

For any angle $\theta \neq \theta_{0}$, we define $L_{\theta}=R_{\theta-\theta_{0}}(L_{\theta_{0}})$. Notice that, for any angle $\theta$, the set $L_{\theta}$ is invariant under the action $\varphi^{A}_{|G_{\theta}}$. Indeed $G_{\theta}=R_{\theta-\theta_{0}} G_{\theta_{0}} R_{\theta-\theta_{0}}^{-1}$ which implies that $\varphi^{A}(G_{\theta})=R_{\theta-\theta_{0}} \varphi^{A}(G_{\theta_{0}}) R_{\theta-\theta_{0}}^{-1}$ and, as the set $L_{\theta_{0}}$ is invariant under the action $\varphi^{A}_{|G_{\theta_{0}}}$, the claim follows. Moreover, for an orientation-preserving homeomorphism $f$ of the circle which sends an angle $\theta$ to another angle $\theta'$, the homeomorphism $\varphi(f)$ sends the set $L_{\theta}$ to the set $L_{\theta'}$. To prove this last fact, use that the homeomorphism $f$ can be written as the composition of a homeomorphism which fixes the point $\theta$ with a rotation.

\begin{lemma} \label{pairdis}
The sets $L_{\theta}$ are pairwise disjoint. Moreover, there exists a homeomorphism from the compact interval $[0,1]$ to the closure of $L_{\theta_{0}}$ which sends the interval $(0,1)$ to the set $L_{\theta_{0}}$.
\end{lemma}

\begin{proof}
By using rotations, we see that it suffices to prove that, for any angle $\theta \neq \theta_{0}$, $L_{\theta_{0}} \cap L_{\theta}= \emptyset$. Fix such an angle $\theta$. Remember that the restriction of the action $\psi'$ to $G_{\theta} \cap G_{\theta_{0}}$ has one fixed point $e_{\theta}$ and is transitive on each connected component of $I- \left\{ e_{\theta} \right\}$. Denote by $E_{\theta}$ the subset of $I$ consisting of prime ends whose only point in the principal set belongs to $\pi^{-1}(L_{\theta})$. As this set is invariant under $\psi_{|G_{\theta} \cap G_{\theta_{0}}}$, it is either empty, or the one-point set $\left\{ e_{\theta} \right\}$ or one of the connected components of $I- \left\{ e_{\theta} \right\}$ or the closure of one of these components or $I$. In the last case, we would have $\widetilde{L_{\theta_{0}}} \subset \pi^{-1}(L_{\theta})$ and $L_{\theta_{0}} \subset L_{\theta}$. Using homeomorphisms of the circle which fix the angle $\theta_{0}$ and send the angle $\theta$ to another angle $\theta'$, we see that, for any angle $\theta'$, $L_{\theta_{0}} \subset L_{\theta'}$. This is not possible as the intersection of the closure of the set $L_{\theta_{0}}$ with the boundary of $A$ is a two-point set which should be invariant under any rotation. The case where the set $E_{\theta}$ is a half-line (open or closed) leads to a similar contradiction by looking at one of the limit sets of $L_{\theta_{0}}$. Hence, for any angle $\theta \neq \theta_{0}$, the intersection $L_{\theta_{0}} \cap L_{\theta}=L_{\theta_{0}} \cap R_{\theta-\theta_{0}}(L_{\theta_{0}})$ contains at most one point, the point $x_{\theta}=\pi(\tilde{x}_{\theta})$.

This implies that any leaf of the form $\left\{ r \right\} \times \mathbb{S}^{1} \subset A$ contains at most two points of $L_{\theta_{0}}$. We claim that if one of these leaves contains two points of $L_{\theta_{0}}$, then no leaf contains exactly one point of this set.

Take a point $x_{\theta_{1}}$ in $L_{\theta_{0}}$ which belongs to the leaf $\left\{ r \right\} \times \mathbb{S}^{1}$. Suppose that there exists another point of the set $L_{\theta_{0}}$ on this leaf. This implies that the point $x_{\theta_{1}}$ belongs to $L_{\theta_{1}}$. Using homeomorphisms in $G_{\theta_{0}}$ which send the point $\theta_{1}$ to another point $\theta \neq \theta_{0}$ of the circle, we see that, for any angle $\theta \neq \theta_{0}$ the point $x_{\theta}$ belongs to $L_{\theta}$. Hence, each circular leaf which meets the set $L_{\theta_{0}}$ contains exactly two points of this set.

If $x$ is another point of $L_{\theta_{0}}$ on the leaf $\left\{ r \right\} \times \mathbb{S}^{1}$, then there exists $\alpha \in \mathbb{S}^{1} - \left\{ 0 \right\}$ such that $R_{\alpha}(x)=x_{\theta_{1}} \in L_{\alpha + \theta_{0}}$. Therefore, $\alpha+ \theta_{0}=\theta_{1}$ and the point $x$ is necessarily the point $R_{\theta_{0}-\theta_{1}}(x_{\theta_{1}})=x_{2\theta_{0}-\theta_{1}}$. Hence, the map $\theta \mapsto p_{1}(x_{\theta})$, where $p_{1}: [r_{1},r_{2}] \times \mathbb{R} \mapsto [r_{1},r_{2}]$ is the projection, is strictly monotonous (as it is one-to-one) on $(\theta_{0}, \theta_{0} + \frac{1}{2}]$ and on $[\theta_{0} + \frac{1}{2}, \theta_{0}+1)$. Moreover, if this map was not globally one-to-one, the circular leaf which contains the point $x_{\theta_{0} + \frac{1}{2}}$ would contain only one point, which is not possible. Hence, this map is strictly monotonous. This implies also that the sets $L_{\theta}$, for $\theta \neq \theta_{0}$, are disjoint from $L_{\theta_{0}}$. The monotonicity of the map $\theta \mapsto p_{1}(x_{\theta})$, combined with Lemma \ref{continuity} implies the second part of the lemma.
\end{proof}

Now, we can complete the proof of the proposition. By Theorem \ref{actionsdersurr}, for any $\theta \in \mathbb{S}^{1}$ and any $r \in (0,1)$ the action $\varphi^{A}_{|G_{\theta}\cap G_{\theta+r}}$ admits a unique fixed point $a(r,\theta)$ in $L_{\theta}$. Moreover, for any point $\theta$ of the circle, the map 
$$ \begin{array}{rcl}
(0,1) & \rightarrow & A \\
r & \mapsto & a(r,\theta)
\end{array}
$$
is one-to-one, continuous and extends to a continuous map $[0,1] \rightarrow A$ which allows us to define $a(0,\theta)$, which is the fixed point of $\varphi^{A}_{|G_{\theta}}$ on one boundary component of $A$, and $a(1,\theta)$, which is the fixed point of $\varphi^{A}_{|G_{\theta}}$ on the other boundary component of $A$. By Lemma \ref{pairdis}, the map $a$ is one-to-one. Moreover, for every point $\alpha$ of the circle, and every point $(r,\theta)$ of the closed annulus $\mathbb{A}= [0,1] \times \mathbb{S}^{1}$, $a(r, \theta+\alpha)=R_{\alpha}(a(r,\theta))$. Indeed, the point $R_{\alpha}(a(r,\theta))$ is the only fixed point of $\varphi^{A}_{|G_{\theta+\alpha}\cap G_{\theta+\alpha+r}}= R_{\alpha}\varphi^{A}_{|G_{\theta}\cap G_{\theta+r}}R_{\alpha}^{-1}$ in $L_{\theta+\alpha}=R_{\alpha}(L_{\theta})$. This implies that the map $a$ is continuous: it is a homeomorphism $\mathbb{A} \rightarrow A$. It remains to prove that it defines a conjugacy with $a_{+}$. Take a homeomorphism $f$, and a point $(r,\theta) \in \mathbb{A}$. Notice that 
$$f G_{\theta}\cap G_{\theta+r} f^{-1}=G_{f(\theta)}\cap G_{f(\theta+r)}=G_{f(\theta)}\cap G_{f(\theta)+(\tilde{f}(\theta+r)-\tilde{f}(\theta))}$$
and that $\varphi^{A}(f)(L_{\theta})=L_{f(\theta)}$.
Therefore, $\varphi^{A}(f)(a(r,\theta))=a(a_{+}(f)(r,\theta))$. The action $\varphi^{A}$ is conjugate to $a_{+}$. As the actions $a_{+}$ and $a_{-}$ are conjugate both by an orientation-preserving homeomorphism and an orientation-reversing homeomorphism, the proposition is proved.
\end{proof}

\subsection{Global conjugacy}

This section is devoted to the end of the proof of Theorem \ref{cerclesursurfaces} in the case of actions on an annulus.

By Lemma \ref{fol2}, our action $\varphi$, restricted to the union $K\times \mathbb{S}^{1}$ of the $F_{\theta}$, where $\theta$ varies over the circle, is conjugate to the restriction of $\varphi_{K,\lambda}$ to $K \times \mathbb{S}^{1}$ by a homeomorphism of the form $(r,\theta) \mapsto (r,\eta(r)+\theta)$ for any $\lambda$. Conjugating $\varphi$ by a homeomorphism of the annulus of the form $(r,\theta) \mapsto (r, \hat{\eta}(r)+\theta)$, where $\hat{\eta}$ is a continuous function which is equal to $\eta$ on $K$, we may assume from now on that the action $\varphi$ is equal to $\varphi_{K,\lambda}$, for any $\lambda$, on $K\times \mathbb{S}^{1}$. 

Moreover, for each connected component $(r_{1}^{i},r_{2}^{i})\times \mathbb{S}^{1}$ of the complement of the closed set $K\times \mathbb{S}^{1}$, the restriction of this action to $[r_{1}^{i},r_{2}^{i}]\times \mathbb{S}^{1}$ is conjugate to $a_{+}$ via an orientation-preserving homeomorphism $g_{i} : [0,1] \times \mathbb{S}^{1} \rightarrow [r_{1}^{i},r_{2}^{i}]\times \mathbb{S}^{1}$, by Proposition \ref{outside}.

We now define a particular continuous function $\lambda_{0}: [0,1] - K \rightarrow \left\{ -1,+1 \right\}$ such that the action $\varphi$ is conjugate to $\varphi_{K,\lambda_{0}}$. For any index $i$, let us denote by $d^{+}_{i}$ the diameter of the set $g_{i}(\left\{ (r,0), 0 \leq r \leq 1 \right\})$ and by $d^{-}_{i}$ the diameter of the set $g_{i}(\left\{ (r,-r), 0 \leq r \leq 1 \right\})$. If $d_{i}^{+} \leq d_{i}^{-}$, we define $\lambda_{0}$ to be identically equal to $+1$ on the interval $(r_{1}^{i}, r_{2}^{i})$. Otherwise, \emph{i.e.} if $d_{i}^{+}> d_{i}^{-}$, we define the map $\lambda_{0}$ to be identically equal to $-1$ on the interval $(r_{1}^{i}, r_{2}^{i})$.

Now, let us define a conjugacy between the action $\varphi$ and the action $\varphi_{K,\lambda_{0}}$. Notice that, by what is above, for any connected component $(r_{1}^{i},r_{2}^{i}) \times \mathbb{S}^{1}$ of the complement of $K \times \mathbb{S}^{1}$, there exists an orientation-preserving homeomorphism $h_{i}$ of $[r_{1}^{i},r_{2}^{i}] \times \mathbb{S}^{1}$ such that, for any orientation-preserving homeomorphism $f$ of the circle, 
$$h_{i} \circ \varphi_{K,\lambda_{0}}(f)_{|[r_{1}^{i},r_{2}^{i}] \times \mathbb{S}^{1}}=\varphi(f)_{|[r_{1}^{i},r_{2}^{i}] \times \mathbb{S}^{1}} \circ h_{i}.$$
Then, for any point $(r,\theta) \in [r_{1}^{i},r_{2}^{i}] \times \mathbb{S}^{1}$ and any angle $\alpha \in \mathbb{S}^{1}$, $h_{i}(r,\theta+\alpha)= h_{i}(r,\theta)+(0,\alpha).$ (Recall that we supposed at the beginning of this section that the morphism $\varphi$ sent the rotation of the circle of angle $\alpha$ on the rotation of the annulus of angle $\alpha$.) 

Now, let us denote by $h: \mathbb{A} \rightarrow \mathbb{A}$ the map which is equal to the identity on $K \times \mathbb{S}^{1}$ and which is equal to $h_{i}$ on the connected component $(r_{1}^{i},r_{2}^{i}) \times \mathbb{S}^{1}$ of the complement of $K \times \mathbb{S}^{1}$. It is clear that the map $h$ is a bijection and that, for a homeomorphism $f$ in $\mathrm{Homeo}_{0}(\mathbb{S}^{1})$, the relation $h \circ \varphi_{K, \lambda_{0}}(f)= \varphi(f) \circ h$ holds. Moreover, the map $h$ commutes with the rotations.

It remains to prove that the map $h$ is continuous. As the map $h$ commutes with rotations, it suffices to prove that the map
$$\begin{array}{rrcl}
\eta: & [0,1] & \rightarrow & \mathbb{A} \\
 & r & \mapsto & h(r,0)
\end{array}
$$
is continuous. Notice first that, for a connected component $(r_{1}^{i},r_{2}^{i})$ of the complement of $K$, the limit $\lim_{ r \rightarrow r_{1}^{i+}} \eta(r)$ is equal to $h_{i}(r_{1},0)$. This last point is the only point of $F_{0}$ on $\left\{ r_{1} \right\} \times \mathbb{S}^{1}$ and is therefore equal to $\eta(r_{1})=(r_{1},0)$. The map $\eta$ is also continuous on the left at $r_{2}^{i}$. It suffices now to establish that, for any sequence $((r_{1}^{n},r_{2}^{n}))_{n \in \mathbb{N}}$ of connected components of the complement of $K$ such that the sequence $(r_{1}^{n})_{n \in \mathbb{N}}$ is monotonous and converges to a point $r_{\infty}$ which belongs to the compact set $K$, for any sequence $(r_{n})_{n \in \mathbb{N}}$ of real numbers such that, for any integer $n$, $r_{n} \in (r_{1}^{n},r_{2}^{n})$, the sequence $(\eta(r_{n}))_{n \in \mathbb{N}}$ converges to $\eta(r_{\infty})=(r_{\infty},0)$. Fix now such a sequence $((r_{1}^{n},r_{2}^{n}))_{n \in \mathbb{N}}$.

\begin{lemma}
For $n$ sufficiently large, one of the curves $h([r_{1}^{n},r_{2}^{n}] \times \left\{ 0 \right\})$ and $h( \left\{ (r, \theta-\frac{r-r_{1}^{n}}{r_{2}^{n}-r_{1}^{n}}), r_{1}^{n} \leq r \leq r_{2}^{n} \right\})$ is homotopic with fixed extremities in the closed annulus $[r_{1}^{n},r_{2}^{n}] \times \mathbb{S}^{1}$ to the curve $[r_{1}^{n},r_{2}^{n}] \times \left\{ 0 \right\}$. We denote by $c_{n}$ this curve.
\end{lemma}

\begin{proof}
Suppose by contradiction that there exists a strictly increasing map $\sigma : \mathbb{N} \rightarrow \mathbb{N}$ such that, for any $n \in \mathbb{N}$, neither $h([r_{1}^{\sigma(n)},r_{2}^{\sigma(n)}] \times \left\{ 0 \right\})$ nor $h( \left\{ (r, \theta-\frac{r-r_{1}^{\sigma(n)}}{r_{2}^{\sigma(n)}-r_{1}^{\sigma(n)}}), r_{1}^{\sigma(n)} \leq r \leq r_{2}^{\sigma(n)} \right\})$ are homotopic to $[r_{1}^{n},r_{2}^{n}] \times \left\{ 0 \right\}$. Then, one of the two curves $h([r_{1}^{\sigma(n)},r_{2}^{\sigma(n)}] \times \left\{ 0 \right\})$ and $h( \left\{ (r, \theta-\frac{r-r_{1}^{\sigma(n)}}{r_{2}^{\sigma(n)}-r_{1}^{\sigma(n)}}), r_{1}^{\sigma(n)} \leq r \leq r_{2}^{\sigma(n)} \right\})$, which we denote by $\gamma_{n}$, admits a lift $\widetilde{\gamma_{n}}$ such that $\widetilde{\gamma_{n}}\cap \left\{ r_{1}^{\sigma(n)} \right\} \times \mathbb{R}= \left\{ (r_{1}^{\sigma(n)},0) \right\}$ and $\widetilde{\gamma_{n}}\cap \left\{ r_{2}^{\sigma(n)} \right\} \times \mathbb{R}= \left\{ (r_{2}^{\sigma(n)},k_{n}) \right\}$, where $k_{n}$ is an integer with $\left|k_{n} \right| \geq 2$. Taking a subsequence if necessary, we can suppose that either, for any $n$, $k_{n} \geq 2$ or, for any n, $k_{n} \leq -2$. To simplify notation, suppose that $k_{n}\geq 2$ for any $n$. 

We now need an intermediate result. Let $(\tilde{x}_{n})_{n \in \mathbb{N}}$ and $(\tilde{y}_{n})_{n \in \mathbb{N}}$ be sequences of points of the band $[0,1] \times \mathbb{R}$ converging respectively to the points $\tilde{x}_{\infty}$ and $\tilde{y}_{\infty}$ such that, for any $n$, the points $\tilde{x}_{n}$ and $\tilde{y}_{n}$ belong to the curve $\widetilde{\gamma_{n}}$ and are not endpoints of this curve. We suppose that $\tilde{x}_{\infty} \neq \tilde{y}_{\infty}$ and that the points $\tilde{x}_{\infty}$ and $\tilde{y}_{\infty}$ are not limit points of the endpoints of $\widetilde{\gamma_{n}}$. For any integer $n$, there exist unique real numbers $r_{n}$ and $r'_{n} \in (0,1)$ such that the point $x_{n}=\pi(\tilde{x}_{n})$ (respectively $y_{n}=\pi(\tilde{y}_{n})$) is the unique fixed point of the group $\varphi(G_{0} \cap G_{r_{n}})$ (respectively $\varphi(G_{0} \cap G_{r'_{n}})$) on the curve $\gamma_{n}$. We claim that, for any strictly increasing map $s:\mathbb{N} \rightarrow \mathbb{N}$ such that the sequences $(r_{s(n)})_{n \in \mathbb{N}}$ and $(r'_{s(n)})_{n \in \mathbb{N}}$ converge respectively to $R$ and $R'$, then $R \neq R'$. Moreover,  the real numbers $R$ as $R'$ are different from $0$ or $1$. 

Let us begin by proving this claim. Suppose by contradiction that there exists a map $s$ such that $R=R'$. We may suppose (by extracting a subsequence and by changing the roles of $x_{n}$ and $y_{n}$ if necessary) that, for any n, $r_{s(n)} < r'_{s(n)}$. Then the set of points in $\gamma_{s(n)}$ which are fixed under one of the actions $\varphi_{|G_{0} \cap G_{r}}$, with $r_{s(n)} \leq r \leq r'_{s(n)}$ (this is also the projection of the set of points on $\widetilde{\gamma_{s(n)}}$ between $\tilde{x}_{s(n)}$ and $\tilde{y}_{s(n)}$) defines a sequence of paths which converge to an interval contained in $\left\{ r_{\infty} \right\} \times \mathbb{S}^{1}$. This interval is the projection of the interval whose endpoints are $\tilde{x}_{\infty}$ and $\tilde{y}_{\infty}$: it has non-empty interior and is necessarily pointwise fixed by $\varphi_{|G_{0} \cap G_{R}}$. Indeed, any homeomorphism $f$ in $G_{0}\cap G_{R}$ fixes the points between $r_{s(n)}$ and $r'_{s(n)}$ for $n$ sufficiently large so that the homeomorphism $\varphi(f)$ fixes the projection of the set of points on $\widetilde{\gamma_{s(n)}}$ between $\tilde{x}_{s(n)}$ and $\tilde{y}_{s(n)}$, for $n$ sufficiently large. Hence, it also pointwise fixes the projection of the interval between the points $\tilde{x}_{\infty}$ and $\tilde{y}_{\infty}$ on the line $\left\{r_{\infty} \right\} \times \mathbb{R}$. However, there exist no such interval with non-empty interior, as the group $\varphi(G_{0} \cap G_{R})$ does not pointwise fix an interval of $\left\{ r_{\infty} \right\} \times \mathbb{S}^{1}$ with non-empty interior, a contradiction.

Let us establish now for instance that $R \neq 1$. The set of points in $\gamma_{s(n)}$ which are fixed by $\varphi_{|G_{0} \cap G_{r}}$, with $r_{s(n)} \leq r \leq 1$ is a path which admits a lift which joins the point $\tilde{x}_{s(n)}$ to either the point $(r_{2}^{\sigma(s(n))},k)$ or the point $(r_{1}^{\sigma(s(n))},k)$. Taking a subsequence if necessary, this sequence of sets converges to an interval with non-empty interior contained in $\left\{ r_{\infty} \right\} \times \mathbb{S}^{1}$. This interval is necessarily pointwise fixed by $\varphi_{|G_{0} \cap G_{R}}$: this is a contradiction.

Let us come back to our proof. There exists a sequence $(\tilde{x}_{n})_{n \in \mathbb{N}}$ of points of the band $[0,1] \times \mathbb{R}$ converging to the point $(r_{\infty},1)$ such that, for any $n$, the point $\tilde{x}_{n}$ belongs to the curve $\widetilde{\gamma_{n}}$ and is not an endpoint of this curve. Let $x_{n}= \pi(\tilde{x}_{n})$. For any integer $n$, there exists a unique real number $r_{n} \in (0,1)$ such that the point $x_{n}$ is the unique fixed point of the group $\varphi(G_{0} \cap G_{r_{n}})$ on the curve $\gamma_{n}$. For any $n$, take a point $\tilde{y}_{n}$ on the curve $\widetilde{\gamma_{n}}$ such that the sequence $(\tilde{y}_{n})_{n}$ converges to a point $\tilde{p}$ which is not a point of the form $(r_{\infty},k)$, where $k$ is an integer. Then it is not a limit point of a sequence of endpoints of the curve $\widetilde{\gamma_{n}}$. As above, we denote by $r'_{n} \in (0,1)$ the real number associated to the point $y_{n}= \pi(\tilde{y}_{n})$. Consider a strictly increasing map $s:\mathbb{N} \rightarrow \mathbb{N}$ such that the sequences $(r_{s(n)})_{n \in \mathbb{N}}$ and $(r'_{s(n)})_{n \in \mathbb{N}}$ converge respectively to $R$ and $R'$. Take a homeomorphism $f$ in $G_{0}$ which sends the point $R$ of the circle to the point $R'$ of the circle. Let $z_{n}=\varphi(f)(x_{n})$. Then, for any $n$, the real number associated to the point $z_{s(n)}$ is necessarily $f(r_{s(n)})$ and the sequence $(f(r_{s(n)}))_{n}$ converges to $R'$. By the claim above, the sequence $(z_{s(n)})_{n}$ has the same limit as the sequence $(y_{s(n)})_{n}$. By continuity, this homeomorphism of the annulus sends the point $(r_{\infty},0)$ to the point $p=\pi(\tilde{p})$. This is not possible as the point $(r_{\infty},0)$ of the annulus is fixed by the group $\varphi(G_{0})$.
\end{proof}

\begin{lemma}
The diameter of the curve $c_{n}$ tends to $0$ as $n$ tends to $+\infty$.
\end{lemma}

\begin{proof}
This proof is similar to the proof of the above lemma. Suppose that the sequence of diameters of $c_{n}$ does not converge to $0$. Let us denote by $\tilde{c}_{n}$ the lift of the curve $c_{n}$ whose origin is the point $(r_{1}^{\sigma(n)},0)$. There exists a subsequence $\tilde{c}_{s(n)}$ of this sequence which converges (for the Hausdorff topology) to an interval with non-empty interior. As the projection of this interval on the annulus is invariant under the action $\varphi_{|G_{0}}$, this compact set projects onto the whole circle $\left\{ r_{\infty} \right\} \times \mathbb{S}^{1}$. Therefore, there exists a sequence $(\tilde{x}_{n})_{n \in \mathbb{N}}$, where the point $\tilde{x}_{n}$ belongs to $\tilde{c}_{s(n)}$, which converges to a point of the form $(r_{\infty}, k)$ with $k \neq 0$ (which is not the limit of a sequence of endpoints of $\tilde{c}_{s(n)}$). There also exists a sequence $(\tilde{y}_{n})_{n \in \mathbb{N}}$, where the point $\tilde{x}_{n}$ belongs to $\tilde{c}_{s(n)}$, which converges to a point $\tilde{p}$ which is not of the form $(r_{\infty}, k)$, where $k$ is an integer. As in the proof of the above lemma, let us denote $r_{n} \in (0,1)$ (respectively $r'_{n}$) the real number associated to the point $x_{n}=\pi(\tilde{x}_{n})$ (respectively $y_{n}=\pi(\tilde{y}_{n})$). Taking a subsequences if necessary, we may suppose that the sequences $(r_{n})_{n}$ and $(r'_{n})_{n}$ converge respectively to the real numbers $R$ and $R'$. As in the proof of the above lemma, $R \neq R'$ and both these numbers are different from $0$ and $1$. Take then a homeomorphism $f$ in $G_{0}$ which sends the point $R$ to the point $R'$. Then the homeomorphism $\varphi(f)$ sends the point $(r_{\infty},0)$ to the point $p=\pi(\tilde{p}) \neq (r_{\infty},0)$, which is not possible as the point $(r_{\infty},0)$ is fixed by the group $\varphi(G_{0})$.
\end{proof}

Now, let us prove that $c_{n}=h([r_{1}^{n},r_{2}^{n}] \times \left\{ 0 \right\})$ for $n$ sufficiently large, which proves the continuity of the map $\eta$ by the above lemma (because the endpoints of these curves converge to the point $\eta(r_{\infty}$. 
Let us denote by $c'_{n}$ the curve among $h([r_{1}^{n},r_{2}^{n}] \times \left\{ 0 \right\})$ and $h( \left\{ (r,\theta-\frac{r-r_{1}^{n}}{r_{2}^{n}-r_{1}^{n}}), r_{1}^{n} \leq r \leq r_{2}^{n} \right\})$ which is not equal to $c_{n}$. As the homotopy class of this curve $c'_{n}$ is not the homotopy class of $[r_{1}^{n},r_{2}^{n}] \times \left\{ 0 \right\}$, the diameter of $c'_{n}$ is bounded from below by $\frac{1}{2}$. Hence, for $n$ sufficiently large (say for $n \geq N$), the diameter of $c_{n}$ is smaller than the diameter of $c'_{n}$.

Fix now $n \geq N$. The orientation of the circle defines an order on $\mathbb{S}^{1}- \left\{ 0 \right\}$. Take a homeomorphism $f\neq Id$ in $G_{0}$ such that, for any $\theta \neq 0$, $f(\theta) \geq \theta$. Then the restriction of $\varphi(f)$ to $h([r_{1}^{n},r_{2}^{n}] \times \left\{ 0 \right\})$ and to $h( \left\{ (r,\theta-\frac{r-r_{1}^{n}}{r_{2}^{n}-r_{1}^{n}}), r_{1}^{n} \leq r \leq r_{2}^{n} \right\})$ defines on orientation on both curves (the orientation such that an arc from a point $x$ to its image by $\varphi(f)$ is positively oriented). If $d_{n}^{+} \leq d_{n}^{-}$, the interior of the curve $c_{n}$ is equal to the curve $g_{n}((0,1) \times \left\{ 0 \right\})$: this is the only simple curve in $(r_{1}^{n},r_{2}^{n}) \times \mathbb{S}^{1}$ which is invariant under $\varphi_{|G_{0}}$ and which is oriented from the point $(r_{1}^{n},0)$ to the point $(r_{2}^{n},0)$. Moreover, the only simple curve in $(r_{1}^{n},r_{2}^{n}) \times \mathbb{S}^{1}$ which is invariant under the action $\varphi_{K,\lambda_{0}}$ restricted to $G_{0}$ which is oriented from the point $(r_{1}^{n},0)$ to the point $(r_{2}^{n},0)$ is the curve $(r_{1}^{n},r_{2}^{n}) \times \left\{ 0 \right\}$. Finally, as the homeomorphism $h$ is continuous on $[r_{1}^{n},r_{2}^{n}] \times \mathbb{S}^{1}$ and $h \circ \varphi_{K,\lambda_{0}}= \varphi \circ h$, we have $h([r_{1}^{n},r_{2}^{n}] \times \left\{ 0 \right\})= c_{n}$. If $d_{n}^{-}<d_{n}^{+}$, the interior of the curve $c_{n}$ is equal to the curve $g_{n}(\left\{ (r,-r), 0 < r < 1 \right\})$: this is the only simple curve in $(r_{1}^{n},r_{2}^{n}) \times \mathbb{S}^{1}$ which is invariant under $\varphi_{|G_{0}}$ and which is oriented from the point $(r_{2}^{n},0)$ to the point $(r_{1}^{n},0)$. Moreover, the only simple curve in $(r_{1}^{n},r_{2}^{n}) \times \mathbb{S}^{1}$ which is invariant under the action $\varphi_{K,\lambda_{0}}$ restricted to $G_{0}$ which is oriented from the point $(r_{2}^{n},0)$ to the point $(r_{1}^{n},0)$ is the curve $(r_{1}^{n},r_{2}^{n}) \times \left\{ 0 \right\}$. Moreover, as the homeomorphism $h$ is continuous on $[r_{1}^{n},r_{2}^{n}] \times \mathbb{S}^{1}$ and $h \circ \varphi_{K,\lambda_{0}}= \varphi \circ h$, we have $h([r_{1}^{n},r_{2}^{n}] \times \left\{ 0 \right\})= c_{n}$.

\section{Case of the torus}

Let $\varphi: \mathrm{Homeo}_{0}(\mathbb{S}^{1}) \rightarrow \mathrm{Homeo}_{0}(\mathbb{T}^{2})$ be a one-to-one group morphism (if such a group morphism is not one-to-one, it is trivial). Recall that such a morphism is continuous by Theorem 4 and Proposition 2 in \cite{RS}. In this section, we prove the following theorem and its corollary:

\begin{theorem}\label{fixedpoint}
For any point $x$ of the circle, the image under the morphism $\varphi$ of the group $G_{x}$ admits a global fixed point.
\end{theorem}

\begin{corollary}
The action $\varphi$ admits an invariant essential circle: the study of this action reduces to the study of an action on an annulus.
\end{corollary}

Here, essential means non-separating so that the surface obtained by cutting along this curve is an annulus. Using the first part of Theorem \ref{cerclesursurfaces} which was proved in last section, this corollary implies directly the second part of this theorem. Let us first see why this theorem implies the corollary.

\begin{proof}[Proof of the corollary]
First, the image under the morphism $\varphi$ of the group $\mathbb{S}^{1}$ of rotations of the circle is conjugate to the subgroup of rotations of the torus of the form:
$$\begin{array}{rcl}
\mathbb{T}^{2}= \mathbb{S}^{1} \times \mathbb{S}^{1} & \rightarrow & \mathbb{T}^{2}\\
(\theta_{1},\theta_{2}) & \mapsto & (\theta_{1}, \theta_{2}+\alpha)
\end{array},$$
where $\alpha \in \mathbb{S}^{1}$. Therefore, after possibly conjugating, we may suppose that the morphism $\varphi$ sends the rotation of the circle of angle $\alpha$ on the rotation $(\theta_{1},\theta_{2}) \mapsto (\theta_{1}, \theta_{2}+\alpha)$.

Fix a point $x_{0}$ on the circle. If $p$ is a fixed point for the group $\varphi(G_{x_{0}})$, then the essential circle $\left\{ \varphi(R_{\alpha})(p), \alpha \in \mathbb{S}^{1} \right\}$ is invariant under the action $\varphi$. The proof of this last claim is similar to the proof of Lemma \ref{fol}.
\end{proof}

\begin{proof}[Proof of Theorem \ref{fixedpoint}]
Let us begin by giving a sketch of the proof of this theorem. First, we prove that the diameters of the images of the fundamental domain $[0,1]^{2}$ under lifts of homeomorphisms in the image of $\varphi$ are uniformly bounded. Then we prove that the restriction of the morphism $\varphi$ to a subgroup of the form $G_{x}$ lifts to a group morphism $\widetilde{\varphi_{x}}: G_{x} \rightarrow \mathrm{Homeo}_{\mathbb{Z}^{2}}(\mathbb{R}^{2})$, where this last group is the group of homeomorphisms of $\mathbb{R}^{2}$ which commute to the integral translations. We prove that the action $\widetilde{\varphi_{x}}$ can be chosen so that it has a bounded orbit. Using these facts, we find a connected subset of $\mathbb{R}^{2}$ with empty interior which is invariant under the action of $\widetilde{\varphi_{x}}$. Using prime ends, we can prove that the action $\varphi_{|G_{x}}$ has a fixed point on the projection of this connected set.

For a homeomorphism $g$ in $\mathrm{Homeo}_{0}(\mathbb{T}^{2})$, we denote by $\tilde{g}$ the lift of $g$ to $\mathrm{Homeo}_{\mathbb{Z}^{2}}(\mathbb{R}^{2})$  (this means that $\pi \circ \tilde{g}= g \circ \pi$ where $\pi: \mathbb{R}^{2} \rightarrow \mathbb{T}^{2}= \mathbb{R}^{2}/\mathbb{Z}^{2}$ is the projection
) with $\tilde{g}(0) \in [-1/2,1/2) \times [-1/2,1/2)$. Let us denote by $D \subset \mathbb{R}^{2}$ the fundamental domain $[0,1]^{2}$ for the action of $\mathbb{Z}^{2}$ on $\mathbb{R}^{2}$.

\begin{lemma}\label{domfond}
The map $\mathrm{Homeo}_{0}(\mathbb{S}^{1}) \rightarrow \mathbb{R}_{+}$ which associates, to any homeomorphism $f$ in $\mathrm{Homeo}_{0}(\mathbb{S}^{1})$, the diameter of the image under $\widetilde{\varphi(f)}$ (or equivalently under any lift of $\varphi(f)$) of the fundamental domain $D$ is bounded.
\end{lemma}

\begin{proof}
The proof of this lemma is almost identical to the proof of Lemma \ref{domfondann}.
\end{proof}

Let $x_{0}$ be a point of the circle.

\begin{lemma}
There exists a group morphism $\widetilde{\varphi_{x_{0}}} : G_{x_{0}} \rightarrow \mathrm{Homeo}_{\mathbb{Z}^{2}}(\mathbb{R}^{2})$ such that:
\begin{itemize}
\item for any homeomorphism $f$ in $G_{x_{0}}$, $\Pi \circ \widetilde{\varphi_{x_{0}}}(f)= \varphi(f)$, where $\Pi:\mathrm{Homeo}_{\mathbb{Z}^{2}}(\mathbb{R}^{2}) \rightarrow \mathrm{Homeo}_{0}(\mathbb{T}^{2})$ is the projection;
\item the subset $\left\{\widetilde{\varphi_{x_{0}}}(f)(0), f \in G_{x_{0}} \right\}$ is bounded.
\end{itemize}
Moreover, the morphism $\widetilde{\varphi_{x_{0}}}$ is continuous.
\end{lemma}

\begin{proof}
Let $G= \varphi(G_{x_{0}})$. Observe that the map
$$\begin{array}{rcl}
G \times G & \rightarrow & \mathbb{Z}^{2} \\
(f,g) & \mapsto & \widetilde{fg}^{-1} (\tilde{f}(\tilde{g}(0)))
\end{array}$$
defines a 2-cocycle on the group $G$ (see \cite{Bro} or \cite{Ghy} for more about the cohomology of groups). Moreover, by Lemma \ref{domfond}, this cocycle is bounded. However, as the group $G$ is isomorphic to the group $\mathrm{Homeo}_{c}(\mathbb{R})$, the group $H^{2}_{b}(G, \mathbb{Z}^{2})$ is trivial (see \cite{Math2} and \cite{MM}). This implies that there exists a bounded map $b: G \rightarrow \mathbb{Z}^{2}$ such that:
$$ \forall f, g \in G, \ \widetilde{fg}^{-1}(\tilde{f}(\tilde{g}(0)))=b(f)+b(g)-b(fg).$$
It suffices then to take for $\widetilde{\varphi_{x_{0}}}$ the composition of the morphism $\varphi_{|G_{x_{0}}}$ with the morphism 
$$ \begin{array}{rcl}
G & \rightarrow &  \mathrm{Homeo}_{\mathbb{Z}^{2}}(\mathbb{R}^{2})\\
f & \mapsto & \tilde{f}+b(f)
\end{array}
.$$
For this action, the orbit of $0$ is bounded by construction. It suffices now to prove that this action is continuous. As the topological space $G_{x_{0}}$ is contractible and the map $\Pi: \mathrm{Homeo}_{\mathbb{Z}^{2}}(\mathbb{R}^{2}) \rightarrow \mathrm{Homeo}_{0}(\mathbb{T}^{2})$ is a covering, there exists a (unique) continuous map $\eta: G_{x_{0}} \rightarrow  \mathrm{Homeo}_{\mathbb{Z}^{2}}(\mathbb{R}^{2})$ which lifts the map $\varphi_{|G_{x_{0}}}$ and sends the identity to the identity. Then the map 
$$ \begin{array}{rcl}
G_{x_{0}} \times G_{x_{0}} & \rightarrow &  \mathrm{Homeo}_{\mathbb{Z}^{2}}(\mathbb{R}^{2}) \\
(f,g) & \rightarrow & \eta(fg)^{-1}\eta(f) \eta(g)
\end{array}
$$
is continuous and its image is contained in the discrete space of integral translations: it is constant and the map $\eta$ is a group morphism. Two group morphisms which lift the group morphism $\varphi_{| G_{x_{0}}}$ differ by a group morphism $G_{x_{0}} \rightarrow \mathbb{Z}^{2}$. However, as the group $G_{x_{0}}$ is simple (hence perfect), such a group morphism is trivial and $\eta =\widetilde{\varphi_{x_{0}}}$.
\end{proof}

We can now complete the proof of Theorem \ref{fixedpoint}. Let us denote by $F$ the closure of the set 
$$ \bigcup_{f \in G_{x_{0}}} \widetilde{\varphi_{x_{0}}}(f)((-\infty,0] \times \mathbb{R}).$$
By the two above lemmas, there exists $M>0$ such that $F \subset (-\infty,M] \times \mathbb{R}$. Denote by $U$ the connected component of the complement of $F$ which contains the open subset $(M,+\infty) \times \mathbb{R}$. Denote by $U'$ the image of $U$ under the projection $p_{2}: \mathbb{R} \times \mathbb{R} \rightarrow \mathbb{R} \times \mathbb{R}/\mathbb{Z}$ and by $\psi_{x_{0}}$ the action on the annulus $\mathbb{R} \times \mathbb{R} / \mathbb{Z}$ defined by 
$$ \forall f \in \mathrm{Homeo}_{0}(\mathbb{S}^{1}), \ \psi_{x_{0}}(f) \circ p_{2}=p_{2}\circ \widetilde{\varphi_{x_{0}}}(f).$$
Notice that $p_{1} \circ \psi_{x_{0}}= \varphi_{|G_{x_{0}}}\circ p_{1}$, where $p_{1}:\mathbb{R} \times \mathbb{R} / \mathbb{Z} \rightarrow \mathbb{R} / \mathbb{Z} \times \mathbb{R} / \mathbb{Z}= \mathbb{T}^{2}$.

By construction, the open set $U'$ is invariant under the action $\psi_{x_{0}}$ (a fundamental domain far on the right must be sent in $U'$ by any homeomorphism in the image of $\psi_{x_{0}}$, by the two above lemmas). Now, this action can be extended to the set of prime ends of $U'$, giving a continuous action $\psi$ of the group $G_{x_{0}}$ (which is isomorphic to the group $\mathrm{Homeo}_{c}(\mathbb{R})$) on the topological space of prime ends of $U'$, which is homeomorphic to $\mathbb{S}^{1}$. 

By Proposition \ref{actionsdersurs}, this last action admits a fixed point. Moreover, for any closed interval $I$ whose interior contains the point $x_{0}$, the set of fixed points of this action contains an open interval and hence an accessible prime end. Therefore, the intersection of the set $F_{I}$ of fixed points of the action $\psi_{x_{0}| G_{I}}$ with $[0,M] \times \mathbb{R}/ \mathbb{Z}$ is non-empty. Therefore, the set of fixed points of $\psi_{x_{0}}$, which is the intersection of the $F_{I}$'s, is non-empty. Theorem \ref{fixedpoint} is proved.
\end{proof}

\section{Case of the sphere and of the closed disc}

In this section, we discuss Conjecture \ref{cerclesursphere}. The following proposition is a first step toward this conjecture and was communicated to me by Kathryn Mann.

\begin{proposition}[Mann] \label{Mann}
Fix a morphism $\varphi: \mathrm{Homeo}_{0}(\mathbb{S}^{1}) \rightarrow \mathrm{Homeo}_{0}(\mathbb{S}^{2})$ (respectively $\varphi: \mathrm{Homeo}_{0}(\mathbb{S}^{1}) \rightarrow \mathrm{Homeo}_{0}(\mathbb{D}^{2})$). Then the action $\varphi$ has exactly two global fixed points on the sphere (respectively one global fixed point on the closed disc).
\end{proposition}

\begin{proof}
The case of the disc is almost identical to the case of the sphere and is left to the reader.

Identify the sphere with $\left\{ (x,y,z) \in \mathbb{R}^{3}, \ x^{2}+y^{2}+z^{2}=1 \right\}$.
By a theorem by Kerekjarto (see \cite{Kol}), the restriction of $\psi$ to the group of rotations of the circle $\mathbb{S}^{1} \subset \mathrm{Homeo}_{0}(\mathbb{S}^{1})$ is topologically conjugate to an action of the form:
$$\begin{array}{rcl}
\mathbb{S}^{1} & \rightarrow & \mathrm{Homeo}_{0}(\mathbb{S}^{2}) \\
\theta & \mapsto & (x,y,z) \mapsto (\cos(\theta)x-\sin(\theta)y, \sin(\theta)x + \cos(\theta)y,z)
\end{array}$$
The action of the circle induced by $\varphi$ has hence exactly two fixed points which we denote by $N$ and $S$. We prove now that the set $\left\{N,S \right\}$ is preserved under any element of the image of $\mathrm{Homeo}_{0}(\mathbb{S}^{1})$ under the morphism $\varphi$. Consider the subset $A \subset \mathrm{Homeo}_{0}(\mathbb{S}^{1})$ consisting of homeomorphisms which commute with a nontrivial finite order rotation of the circle. Then any element of $\varphi(A)$ preserves the set of fixed points of the image under $\varphi$ of a nontrivial finite order rotation. This last set is equal to $\left\{ N,S \right\}$. By the following lemma, each element of the group $\varphi(\mathrm{Homeo}_{0}(\mathbb{S}^{1}))$ preserves the set $\left\{N,S \right\}$.

\begin{lemma} \label{generate}
The set $A$ generates the group $\mathrm{Homeo}_{0}(\mathbb{S}^{1})$, \emph{i.e.} any homeomorphism in $\mathrm{Homeo}_{0}(\mathbb{S}^{1})$ can be written as a product of elements of $A$.
\end{lemma}

Now, the action $\psi$ restricted to the set $\left\{N,S \right\}$ induces a morphism $\mathrm{Homeo}_{0}(\mathbb{S}^{1}) \rightarrow \mathbb{Z} / 2$. As the group $\mathrm{Homeo}_{0}(\mathbb{S}^{1})$ is simple, such a morphism is trivial: Proposition \ref{Mann} is proved.
\end{proof}

\begin{proof}[Proof of Lemma \ref{generate}]
By the fragmentation lemma (see \cite{Bou} Theorem 1.2.3), any homeomorphism in $\mathrm{Homeo}_{0}(\mathbb{S}^{1})$ can be written as a product of homeomorphisms each supported in an interval whose length is smaller than $1/6$ (where the length of the circle is equal to $1$). Moreover, any homeomorphism supported in the interior of an interval $I \subsetneq \mathbb{S}^{1}$ can be written as a commutator:
$$f= f_{1}f_{2}f_{1}^{-1}f_{2}^{-1},$$
where $f_{1}$ and $f_{2}$ are homeomorphisms of the circle supported in $I$ (see \cite{Mil} Lemma 4.6). Thus it suffices to prove that any commutator of homeomorphisms supported in a same interval whose length is smaller than $1/6$ can be written as a product of elements of $A$.

Take an interval $I \subset \mathbb{S}^{1}$ whose length is smaller than $1/6$ and let $f_{1}$ and $f_{2}$ be two homeomorphisms supported in $I$. Let us denote by $R_{\theta}$ the rotation of the circle of angle $\theta$. For $i=1,2$, let $g_{i}$  be the homeomorphism defined by $g_{i|I}=f_{i}$ , $g_{i|R_{\frac{1}{2}}(I)}=R_{\frac{1}{2}}f_{i}R_{\frac{1}{2}}^{-1}$, and $g_{i}(x)=x$ elsewhere. Notice that the homeomorphisms $g_{1}$ and $g_{2}$ commute with the rotation $R_{\frac{1}{2}}$ and hence belong to the set $A$. Take a homeomorphism $h$ in $A$ which commutes with the order $3$ rotations such that $h(R_{\frac{1}{2}}(I)) \cap R_{\frac{1}{2}}(I)= \emptyset$ and $h_{|I}=Id_{I}$ (such a homeomorphism $h$ exists as the length of the interval $I$ is small enough). Then the homeomorphism $[g_{1},hg_{2}h^{-1}]$ is equal to $[g_{1},g_{2}]=[f_{1},f_{2}]$ on $I$, to $[Id,hg_{2}h^{-1}]=Id$ on $h(R_{\frac{1}{2}}(I))$, to $[g_{1}, Id]=Id$ on $R_{\frac{1}{2}}(I)$ and to the identity elsewhere. Hence:
$$[f_{1},f_{2}]=[g_{1},hg_{2}h^{-1}]$$
and Lemma \ref{generate} is proved.
\end{proof}

It is natural now to try to adapt the proof of Section 4 to prove Conjecture \ref{cerclesursphere}. As in the case of the annulus, we can find an invariant lamination by circles but there is a problem when this lamination does not accumulate on one of the global fixed point of this action: one has to study the actions of the group of orientation-preserving homeomorphisms of the circle on the open annulus or on the half-open annulus such that the groups of the form $G_{\theta}$ have no fixed point in the interior of these surfaces. If we try to adapt the proof of Subsection 4.2, we are confronted with a problem: Lemma \ref{domfondann} is false in this case (it is easy to find counter-examples) and it seems difficult to adapt it in our situation. However, this lemma seems to be the only problematic step for a proof of this conjecture.

\section*{Acknowledgement}

I would like to thank Frédéric Le Roux and Kathryn Mann for their careful reading of this article. I also thank Kathryn Mann for explaining me Proposition \ref{Mann}.

\end{document}